\documentclass[12pt]{amsart}
\usepackage{fullpage}
\usepackage{times}
\usepackage{latexsym}
\usepackage{amssymb}
\usepackage{pinlabel}
\usepackage[all]{xy}

\newtheorem{thmx}{Theorem}

\newtheorem{thm}{Theorem}[section]
\newtheorem{prop}[thm]{Proposition}
\newtheorem{cor}[thm]{Corollary}

\newtheorem{lem}[thm]{Lemma}

\theoremstyle{definition}
\newtheorem{rem}[thm]{Remark}
\newtheorem{defn}[thm]{Definition}
\newtheorem{ex}[thm]{Example}

\usepackage{hyperref}

\newcommand{\C}{\mathbb{C}}

\newcommand{\Q}{\mathbb{Q}}
\newcommand{\R}{\mathbb{R}}
\newcommand{\Z}{\mathbb{Z}}

\title[Fundamental groups of aspherical manifolds and maps of non-zero degree]
 {Fundamental groups of aspherical manifolds\\ and maps of non-zero degree}

\author{Christoforos Neofytidis}
\address{Department of Mathematical Sciences, {\smaller SUNY} Binghamton, Binghamton, NY 13902-6000, USA}
\email{chrisneo@math.binghamton.edu}
\date{\today; \copyright{\ C. Neofytidis 2015--2017}}

\begin{document}
\maketitle

\begin{abstract}
We define a new class of irreducible groups, called groups {\em not infinite-index presentable by products} or {\em not IIPP}.
We prove that certain aspherical manifolds with fundamental groups not IIPP do not admit maps of non-zero degree from direct products. This extends previous results of Kotschick and L\"oh, providing new classes of aspherical manifolds -- beyond those non-positively curved ones which were predicted by Gromov -- that do not admit maps of non-zero degree from direct products.

A sample application is that an aspherical geometric 4-manifold admits a map of non-zero degree from a direct product if and only if it is a virtual product itself. This completes a characterization of the product geometries due to Hillman.  Along the way we prove that for certain groups the property IIPP is a criterion for reducibility. This especially implies the vanishing of the simplicial volume of the corresponding aspherical manifolds. 
It is shown that aspherical manifolds with reducible fundamental groups do always admit maps of non-zero degree from direct products.  
\end{abstract}

\section{Introduction}

A fundamental topic in topology is the study of maps of non-zero degree between manifolds of the same dimension. The existence of a map of non-zero degree defines a transitive relation, called domination relation, on the homotopy types of closed oriented manifolds. Whenever such a map $M\longrightarrow N$ exists we say that $M$ {\em dominates} $N$ and denote this by $M\geq N$. Gromov suggested investigating the domination relation as defining a partial order of manifolds and formulated several conjectures for candidate classes that might (not) be comparable under $\geq$~\cite{Gromov:metric,Gromov:bounded,Gromov:maps,CarlsonToledo,Wang:3-mfdsasp,KotschickLoeh1}. A particular case of the domination question is when the domain is a non-trivial direct product. That case was raised in Gromov's theory of functorial semi-norms (such as the simplicial volume, see also~\cite{MilnorThurston}) and of topological rigidity, where Gromov predicted that the fundamental classes of certain aspherical manifolds with large universal covers are not dominated by products. Furthermore, the domination-by-products question has its own independent interest, being a special case of Steenrod's classical problem on the realization of homology classes by manifolds~\cite{EilenbergSteenrod}. 

The homotopy types of aspherical manifolds are determined by their fundamental groups. Long-standing rigidity conjectures state that their homeomorphism types are determined by their fundamental groups as well. Several related questions concern the decomposition of the fundamental group of an aspherical manifold as a direct product and the realization of finitely presented Poincar\'e duality groups as fundamental groups of aspherical manifolds. An open question in this context is whether every closed aspherical manifold $M$ with fundamental group $\Gamma_1 \times \Gamma_2$ can be decomposed 
as a product of closed manifolds with fundamental groups $\Gamma_1$ and $\Gamma_2$ respectively. In particular, it is an open question whether aspherical manifolds with {\em reducible} fundamental groups (that is, 
virtual products of two infinite groups) are finitely covered -- and therefore dominated -- by products. L\"uck showed how to obtain an affirmative answer in dimensions higher than four, relying on very strong assumptions concerning the Farrell-Jones conjecture and the cohomological dimensions of the involved groups~\cite{Lue}. For non-positively curved manifolds
an affirmative answer is given by Gromoll-Wolf's isometric splitting theorem~\cite{GromollWolf}. 

\subsection{Summary of results}

Before stating the main results in detail, we give a general overview of the content of this paper. First, applying Thom's emphatic answer (in rational homology)~\cite{Thom} of Steenrod's realization problem, we will show that every closed aspherical manifold with reducible fundamental group is indeed dominated by products (Theorem \ref{thmA}). 
Therefore, in this paper we are interested to investigate whether aspherical closed manifolds with irreducible (i.e. not reducible) fundamental group admit arbitrary maps of non-zero degree by non-trivial direct products (clearly such maps cannot be homotopic to coverings). 

Our goal is to introduce a new algebraic obstruction to domination by products for aspherical manifolds whose fundamental groups have non-trivial center. This property will be termed ``groups not infinite-index presentable by products" or ``not IIPP" for short (cf. Definition \ref{def:notIIPP}). We will obtain new classes of aspherical manifolds that are not dominated by direct products, extending a previous obstruction of Kotschick and L\"oh~\cite{KotschickLoeh1} (called ``groups not presentable by products"), and expanding Gromov's predictions (about irreducible, non-positively curved manifolds~\cite{Gromov:metric}) to aspherical manifolds with non-trivial center. In particular, we will show that large classes of circle bundles with fundamental groups not IIPP are not dominated by products (Theorem \ref{thmB}). In many cases, we will prove that the existence of a map of non-zero degree from a direct product is equivalent to the existence of a finite covering of the same product type (Theorem \ref{thmC}). 

By definition, a group not presentable by products is not IIPP and a group not IIPP is irreducible, but none of those implications can be reversed (cf. Remark \ref{r:nonequivalentconditions}). Nevertheless, the characterization of Theorem \ref{thmC} contains a special case of equivalence between ``irreducible" and ``not IIPP"
for arbitrary (i.e. not necessarily finitely generated or torsion-free) groups: If the quotient of a group $\Gamma$ by its center $C(\Gamma)$ is not presentable by products, then $\Gamma$ is reducible if and only $\Gamma$ is IPPP (Theorem \ref{thmD}). This result fails when $\Gamma/C(\Gamma)$ is presentable by products. For example, the $5$-dimensional Heisenberg group $H_5$ -- whose quotient by its center is $\Z^4$ and thus presentable by products -- is IIPP but irreducible (cf. Example \ref{ex:Heisenberg}). Dimension five is the sharp dimension in which this phenomenon occurs, since ``IIPP" and ``reducible" are equivalent for fundamental groups of aspherical geometric manifolds (in the sense of Thurston) in dimensions $\leq4$; cf. Section \ref{s:proofE}, in particular Theorem \ref{t:4-mfdgroupsIIPP}.

As a sample application of this study we deduce (combining Theorems \ref{thmA} and \ref{thmD}) the vanishing of the simplicial volume of certain aspherical manifolds (Corollary \ref{c:simplicial}).
Moreover, our results combined with Hillman's work~\cite{Hillman} show that a $4$-dimensional geometric manifold 
is dominated by a product if and only if it is covered by a product (Theorem \ref{thmF}).
Further applications about ordering manifolds using maps of non-zero degree~\cite{CarlsonToledo,Wang:3-mfdsasp} and the monotonicity of Kodaira dimensions with respect to the existence of maps of non-zero degree~\cite{Zhang} will be presented in a subsequent paper~\cite{NeoOrder}.

\subsection{Main theorems}

In the development of the theory of bounded cohomology, Gromov~\cite{Gromov:bounded} conjectured that the
fundamental classes of irreducible, locally symmetric spaces of non-compact type cannot be represented by non-trivial products 
(of surfaces). Kotschick and L\"oh~\cite{KotschickLoeh1} verified Gromov's conjecture, by finding an algebraic obstruction to domination by products for rationally essential
manifolds. A closed oriented connected $n$-dimensional manifold $M$ is called {\em rationally essential} if the classifying map of the universal covering $c_M
\colon M \longrightarrow B\pi_1(M)$ sends the fundamental class of $M$ to a non-trivial element in $H_n(B\pi_1(M);\Q)$; see~\cite{Gromov:metric,Gromov:bounded}. The non-domination criterion of~\cite{KotschickLoeh1}, given in Theorem \ref{t:KotschickLoehmain} below, reads as follows: An infinite group $\Gamma$ is called {\em not presentable by products} if, for every
homomorphism $\varphi \colon \Gamma_1 \times \Gamma_2 \longrightarrow \Gamma$ onto a finite index subgroup of $\Gamma$ one of the factors
$\Gamma_i$ has finite image $\varphi(\Gamma_i) \subset \Gamma$.

As we shall see in Section \ref{s:notation}, the proof of Theorem \ref{t:KotschickLoehmain} was obtained by showing that the existence of a map of non-zero degree $f \colon X_1 \times X_2 \longrightarrow M$, with $M$ rationally essential, implies a presentation by products $\Gamma_1 \times \Gamma_2\longrightarrow\pi_1(M)$, where $\Gamma_i := \mathrm{im}(\pi_1(f\bigl\vert_{X_i}))$. However, the proof of that statement does not give any insight on the index of the presenting factors $\Gamma_1$, $\Gamma_2$ in $\pi_1(M)$. In fact, all the possibilities for the indices $[\pi_1(M):\Gamma_i]$ may occur, as we will see in Example \ref{ex:allcasesindex}. The targets of that example are however direct products of (aspherical) manifolds. Hence, a main question is to understand how close to ``reducible" must the fundamental group of an aspherical manifold be, in order this manifold to be dominated by products. Indeed, the case of reducible fundamental groups has a complete affirmative answer:

\begin{thmx}\label{thmA}
A closed aspherical manifold with reducible fundamental group is dominated by a non-trivial direct product of closed oriented manifolds.
\end{thmx}  

Thus our purpose in this article is to examine aspherical manifolds with irreducible fundamental groups. We extend the notion ``group not presentable by products", including groups with (virtually) infinite center: 

\begin{defn}\label{def:notIIPP}
 An infinite group $\Gamma$ is called {\em not infinite-index presentable by products} or {\em not IIPP} if, for every homomorphism $\varphi \colon \Gamma_1 \times \Gamma_2
\longrightarrow \Gamma$ onto a finite index subgroup of $\Gamma$ at least one of the images $\varphi(\Gamma_i)$ has finite index in $\Gamma$. 
\end{defn}

Otherwise, if such homomorphism $\varphi$ exists with $[\Gamma:\varphi(\Gamma_i)]=\infty$ for both $i$, then $\Gamma$ is called {\em infinite-index presentable by products} or {\em IIPP}. 

\begin{rem}\label{r:nonequivalentconditions}
The following inclusions hold immediately by definition:  
\begin{center}
$\{$groups not presentable by products$\}$ $\subset$ $\{$groups not IIPP$\}$ $\subset$ $\{$irreducible groups$\}$. 
\end{center}
However, none of the inverse inclusions hold: For the first inclusion, the infinite cyclic group is a trivial example of a group presentable by products, but not IIPP. For the second inclusion, the $5$-dimensional Heisenberg group is an example of an IIPP irreducible group; cf. Example \ref{ex:Heisenberg}.
\end{rem}

The strong feature of the property ``not IIPP" is that it detects at once all the possible dimensions of the factors of a product that dominates (with a $\pi_1$-surjective map) a rationally essential manifold with torsion-free fundamental group. For aspherical manifolds one of the factors can be taken as simple as possible:

\medskip
{\it Up to finite covers, if an aspherical manifold $M$ with fundamental group not IIPP is dominated by a product $X_1\times X_2$, then $X_1\times T^{\dim X_2}\geq M$, where $\dim X_2\leq\mathrm{rank} C(\pi_1(M))$.}
\medskip

Our first main non-domination result deals with circle bundles: 

\begin{thmx}\label{thmB}
Let $M$ be a circle bundle over a closed oriented aspherical manifold $B$, so that $\pi_1(M)$ is not IIPP and its center remains infinite cyclic in finite covers. Then $M$ is not dominated by any non-trivial direct product of closed oriented manifolds. 
\end{thmx}

\begin{ex}\label{ex:nilpotentnotIIPPexample}
 Closed $Nil^4$-manifolds 
 fulfill the conditions of Theorem \ref{thmB} and therefore are never dominated by products; cf. Propositions
\ref{p:nil4-mfds} and \ref{p:nil4-mfdsnotIIPP}. Examples of such manifolds can easily be constructed as mapping tori of suitable self-homeomorphisms of
$T^3$; see Remark \ref{r:nil4mappingtorus}.
\end{ex}

The fundamental group of $M$ in Theorem \ref{thmB} is presentable by products having infinite (cyclic) center. By Eberlein's works~\cite{Eberlein1,Eberlein2} or, more generally, by the stirring work of Farb and Weinberger~\cite{FarbWeinb}, irreducible 
manifolds of dimension higher than one with fundamental groups with center do not admit metrics of non-positive sectional curvature. The non-domination results of Kotschick and L\"oh~\cite{KotschickLoeh1} deal mostly with non-positively curved manifolds. More precisely, they show that non-positively curved manifolds are dominated by products if and only if they are virtually diffeomorphic to products. Equivalently, the fundamental groups of those manifolds 
are reducible if and only if they are presentable by products. In the case of fiber bundles (which includes manifolds that do not admit 
metrics of non-positive sectional curvature~\cite{KapLeeb}), the results of~\cite{KotschickLoeh1} deal with targets whose fiber and base have fundamental groups not presentable by products and with targets with positive simplicial volume. 

A circle bundle $M$ over an aspherical manifold has vanishing simplicial volume~\cite{Gromov:bounded,Lueck:book}. Moreover, the infinite cyclic fundamental group of the $S^1$-fiber is presentable by products and, as mentioned above, central in $\pi_1(M)$, which means that $\pi_1(M)$ is presentable by products. 
However, as noted in Remark \ref{r:nonequivalentconditions}, the property ``presentable by products" is not generally equivalent to ``reducible", and this applies as well to fundamental groups of many circle bundles; e.g. the fundamental group of any circle bundle over a closed surface with non-zero rational Euler class is irreducible and presentable by products (but it is not IIPP; see Proposition \ref{p:3-mfdsnotIIPP}).  
In fact, when the base of the circle bundle has fundamental group not presentable by products, then $\pi_1(M)$ is reducible if and only if it is IIPP. The following result characterizes circle bundles with the latter property, and is a partial converse 
to Theorem \ref{thmB}:

\begin{thmx}\label{thmC}
 Let $M \stackrel{\pi}\longrightarrow B$ be a circle bundle over a closed aspherical manifold $B$ whose fundamental group $\pi_1(B)$ is not presentable by
products. Then the following are equivalent:
 \begin{itemize}
 \item[(1)] $M$ is dominated by a non-trivial product of closed oriented manifolds;
 \item[(2)] $M$ is finitely covered by a product $S^1 \times B'$, for some finite cover $B' \longrightarrow B$;
 \item[(3)] $\pi_1(M)$ is reducible;
 \item[(4)] $\pi_1(M)$ is IIPP.
 \end{itemize}
\end{thmx}

\begin{ex}
A closed $4$-manifold $M$ carrying the geometry $Sol_1^4$ is virtually a circle bundle over a
closed oriented $Sol^3$-manifold, 
and $\pi_1(M)$ is not IIPP (see Propositions \ref{p:sol14-mfds} and \ref{p:sol14-mfdsIIPP} respectively). Since $Sol^3$-manifold groups are not presentable by products (cf. Proposition \ref{p:3-mfdsPPSeifert}), Theorem \ref{thmC} implies that $M$ is not dominated by products. Actually, a circle bundle over a closed oriented $Sol^3$-manifold is dominated by a product if and only if it possesses the geometry $Sol^3 \times \R$ (see
also Theorem \ref{t:hillmanproducts}).
\end{ex}

The idea of Theorem \ref{thmC} is that 
the center of $\pi_1(B)$ remains trivial in finite covers, being not presentable by products and torsion-free. In fact, the equivalence between $(3)$ and $(4)$ holds in its greatest generality with no assumptions on finite generation, torsion freeness or the virtual rank of the center:

\begin{thmx}\label{thmD}
Let $\Gamma$ be a group with center $C(\Gamma)$ such that the quotient $\Gamma/C(\Gamma)$ is not presentable by products. Then, $\Gamma$ is reducible if and only if it is IIPP.
\end{thmx}

In the light of Theorem \ref{thmA} we obtain the following consequence of Theorem \ref{thmD}:

\begin{cor}\label{c:simplicial}
Let $M$ be a closed oriented aspherical manifold such that $\pi_1(M)$ is IIPP and $\pi_1(M)/C(\pi_1(M))$ is not presentable by products. Then there exists a closed oriented manifold $N$ such that $T^k\times N\geq M$, where $T^k$ is the $k$-dimensional torus with $k<\mathrm{rank} C(\pi_1(M))$. In particular, $M$ has zero simplicial volume.
\end{cor}

\begin{rem}
It is a long-standing conjecture that the simplicial volume of a closed aspherical manifold whose fundamental group contains a non-trivial amenable normal subgroup vanishes~\cite{Lueck:book}. The vanishing result in Corollary \ref{c:simplicial} is a straightforward consequence of the domination $T^k\times N\geq M$ (applying Theorems \ref{thmD} and \ref{thmA}). Alternatively, after showing that $\pi_1(M)$ is a virtual product with an Abelian factor (by Theorem \ref{thmD}), the vanishing of the simplicial volume of $M$ follows as well by Gromov's isometry theorem with respect to the simplicial $\ell^1$-norm~\cite{Gromov:bounded}. 
\end{rem}

The property IIPP truly recognizes reducibility of a group $\Gamma$ whenever the quotient $\Gamma/C(\Gamma)$ is not presentable by products. Indeed, IIPP is not anymore a criterion for a group $\Gamma$ to be reducible when the quotient $\Gamma/C(\Gamma)$ is presentable by products:

\begin{ex}\label{ex:Heisenberg}
The $5$-dimensional Heisenberg group $H_5$ is {\em irreducible and IIPP}, with $H_5/C(H_5)$ isomorphic to $\Z^4$; cf. Lemma \ref{l:H5}. Also, $H_5$ is realizable by a non-trivial circle bundle over $T^4$. In Theorem \ref{t:H5}, we will prove that this manifold is not dominated by products. 
This seems to be the first example of an aspherical manifold whose fundamental group admits a presentation by a product of two non-trivial subgroups of infinite index, but this manifold is still not dominated by products. Previously, examples of rationally essential, but not aspherical, manifolds with reducible fundamental groups that are not dominated by products were given in~\cite{KotschickLoeh1}. 

Also, this example shows that the converse of Theorem \ref{thmB} does not hold in general. In particular, the assumptions of Theorem \ref{thmC} (which is a partial converse of Theorem \ref{thmB}) cannot be removed.
\end{ex}

The Heisenberg group $H_5$ has cohomological dimension five which allows enough space for the existence of two infinite commuting subgroups $\Gamma_i \subset H_5$ such that $H_5$ is presented by $\Gamma_1 \times \Gamma_2$ and $[H_5 : \Gamma_i]=\infty$. This is not true in lower dimensions as we shall see in Theorem \ref{thmE} below. 

\subsection{Examples and applications}

We now characterize the fundamental groups of certain aspherical manifolds, giving non-trivial examples of groups presentable by products, but not IIPP. Then, we apply our main results to show that an aspherical $4$-manifold possessing a Thurston geometry is dominated by a direct product if and only if it is virtually a direct product. 

\subsubsection{Non-trivial examples of groups not IIPP}

The condition ``not IIPP" is the crucial property for the non-existence results of this paper. The following result gathers together some non-elementary examples of groups not IIPP.

\begin{thmx}\label{thmE}
Irreducible fundamental groups of aspherical manifolds that possess a Thurston solvable geometry in dimensions $\leq 4$ are not IIPP.
\end{thmx}

The above statement contains solvable groups that are not presentable by products as well, namely $Sol^3$-, $Sol_0^4$- and $Sol_{m\neq n}^4$-manifold groups.
Theorem \ref{thmE} says, roughly, that in low dimensions the properties ``reducible" and ``IIPP" are actually equivalent. The requirement on the cohomological dimension being at most four is crucial in the above theorem, because as we have seen in Example \ref{ex:Heisenberg} the $5$-dimensional Heisenberg group $H_5$ is irreducible and IIPP. 

\subsubsection{Domination by products for geometric $4$-manifolds}

After~\cite{KotschickLoeh1}, it is natural to ask to what extent the condition ``presentable by products" on the fundamental group of a rationally essential manifold $M$ would be sufficient for $X_1\times X_2\geq M$. Theorem \ref{thmA} says that reducibility suffices for aspherical manifolds. A complete answer is known in dimension three~\cite{KotschickNeofytidis}, where Kotschick and the author proved that a closed $3$-manifold is dominated by products if and only if it is either a virtual product or a virtual connected sum of copies of $S^2\times S^1$. Thus, in particular, non-trivial circle bundles over closed oriented aspherical surfaces are never dominated by products, although their fundamental groups are presentable by products, having infinite center. It is well known~\cite{Thurstonbook,Scott:3-mfds} that a closed $3$-manifold possesses one of the geometries $Nil^3$ or $\widetilde{SL_2}$ if and only if it is virtually a non-trivial circle bundle over a closed aspherical surface (torus or a hyperbolic surface respectively). A main application of Theorems \ref{thmB} and \ref{thmC}, together with the examples of groups not IIPP given in Theorem \ref{thmE}, is the following characterization in dimension four:

\begin{thmx}\label{thmF}
 A closed oriented aspherical geometric $4$-manifold $M$ is dominated by a non-trivial product if and only if it is finitely covered by a product.
Equivalently, $M$ carries one of the product geometries $\mathbb{X}^3 \times \R$ or the reducible $\mathbb{H}^2 \times \mathbb{H}^2$ geometry.
\end{thmx}

The existence of finite coverings of (diffeomorphism) type $N\times S^1$ for manifolds modeled on product geometries $\mathbb{X}^3\times\R$ is due to Hillman~\cite{Hillman} (note that domination by products alone follows by Theorem \ref{thmA} as well). The above theorem says that only geometric aspherical $4$-manifolds that are virtual products admit maps of non-zero degree from direct products. 

\begin{rem}
We note that we could have included non-aspherical geometries as well in the above statement, however those geometries are not interesting for the domination-by-products question, either because they are products themselves or because their representatives are simply connected. The latter geometries were contained as trivial examples in~\cite{Neofytidis}, where we constructed maps from products to every simply connected $4$-manifold.
\end{rem}

\subsection*{Acknowledgements}
Most of the content of this paper stems by a large part of author's thesis, which was defended in Munich. The author is grateful to his advisor, Dieter Kotschick, for his wise guidance and support. Among several other people, the author is particularly grateful to Clara L\"oh and to Shicheng Wang for useful discussions and many suggestions; Shicheng Wang's previous numerous works on maps of non-zero degree have been a great source of inspiration and have strongly influenced this research project. The financial support of the {\em Deutscher Akademischer Austausch Dienst} (DAAD) is also gratefully acknowledged.

\section{Background and Preliminaries}

In this section, we give a short description of the background of the topic of this paper, explain the notation, 
and give some elementary properties of groups presentable by products.

\subsection{The domination relation}

In the early 1940s, Steenrod raised the question of whether every $n$-dimensional integral homology class can be realized as the image of the fundamental
class of a closed oriented $n$-dimensional manifold under a continuous map~\cite[Problem 25]{EilenbergSteenrod}. About a decade later, Thom answered
affirmatively
Steenrod's question in degrees up to six, 
and found a $7$-dimensional integral homology class which is not
realizable by a closed manifold. (Since then, other non-realizability results have been obtained.) Nevertheless, Thom proved that, in all degrees, some multiple of each integral
homology class can be realized by a closed smooth manifold  \cite{Thom}.

In this paper, we are interested in the realization of fundamental classes of closed oriented manifolds (especially by direct products of manifolds), and
therefore we deal with the notion of {\em the degree of a continuous map}. Namely, suppose that $f \colon M \longrightarrow N$ is a continuous map between
two closed oriented $n$-dimensional manifolds. The {\em degree} of $f$ is defined to be the integer $d$ so that $H_n(f;\Z)([M]) = d \cdot [N]$, where
$[M] \in H_n(M;\Z)$ and $[N] \in H_n(N;\Z)$ denote the fundamental classes of $M$ and $N$ respectively. 
Whenever $d$ is not zero, we say that $M$
{\em dominates} $N$ (or that $M$ {\em $d$-dominates} $N$), and write $M \geq N$ (or $M \geq_d N$). The degree of $f$ is denoted by $\deg(f)$. Unless
otherwise stated, in this paper we consider continuous maps between the homotopy types of closed oriented connected manifolds.

Gromov suggested to investigate the domination relation as defining an ordering of manifolds and as a tool to understand the values of
functorial semi-norms on homology, most notably the simplicial volume~\cite{Gromov:metric,Gromov:bounded,CarlsonToledo}. The simplicial volume of a closed manifold $M$ is completely determined by the classifying space of the fundamental group, $B\pi_1(M)$, because the classifying map of the universal covering, $c_M \colon M \longrightarrow B\pi_1(M)$, 
induces an isometry $H_*(c_M;\Q) \colon H_*(M;\Q) \longrightarrow H_*(B\pi_1(M);\Q)$ with respect to the simplicial $\ell^1$-norm~\cite{Gromov:bounded}. This gives rise to the following definition:

\begin{defn}[\cite{Gromov:bounded}]\label{d:ess}
 A closed oriented connected $n$-dimensional manifold $M$ is called {\em rationally essential} if $H_n(c_M;\Q)([M]) \neq 0 \in H_n(B\pi_1(M);\Q)$, where
$c_M \colon M \longrightarrow B\pi_1(M)$ classifies the universal covering of $M$. Otherwise, $M$ is called {\em rationally inessential}.
\end{defn}

Clearly, every closed aspherical manifold is rationally essential. In fact, the notion of essentialness expands widely the class of aspherical manifolds, because for example every connected sum containing a rationally essential summand is rationally essential itself, and every manifold with non-zero simplicial volume is rationally essential
~\cite{Gromov:bounded}.

\subsection{Kotschick-L\"oh's non-domination criterion}\label{s:notation}

Gromov conjectured that there might exist certain classes of (rationally essential) manifolds which are not dominated by products, pointing out irreducible locally symmetric spaces of non-compact type as potential candidates; cf.~\cite[Chapter 5G$_+$]{Gromov:metric}. Kotschick and L\"oh~\cite{KotschickLoeh1} verified Gromov's suggestion, by finding a condition on the fundamental groups of rationally essential manifolds that are dominated by products:

\begin{defn}[\cite{KotschickLoeh1}]\label{d:PP}
An infinite group $\Gamma$ is called {\em presentable by products} if there is a homomorphism $\varphi \colon \Gamma_1 \times \Gamma_2 \longrightarrow
\Gamma$ onto a finite index subgroup of $\Gamma$ so that the restriction of $\varphi$ to each factor $\Gamma_i$ has infinite image $\varphi(\Gamma_i)$. 
\end{defn}

The property of being (not) presentable by products is clearly preserved under passing to finite index subgroups.

\begin{ex}\label{ex:examplesPPtf} \
\begin{itemize}
 \item[\normalfont{(1)}] A {\em reducible group} is obviously presentable by products, being a virtual product of two infinite groups.
 \item[\normalfont{(2)}]  Let $\Gamma$ be a group which contains a finite index subgroup $\overline{\Gamma}$ with infinite center. Then $\Gamma$ is
presentable by products through the multiplication homomorphism $C(\overline{\Gamma}) \times \overline{\Gamma} \longrightarrow \overline{\Gamma}$.
\end{itemize}
These two examples include every torsion-free group presentable by products \cite[Prop. 3.2]{KotschickLoeh1}.
\end{ex}

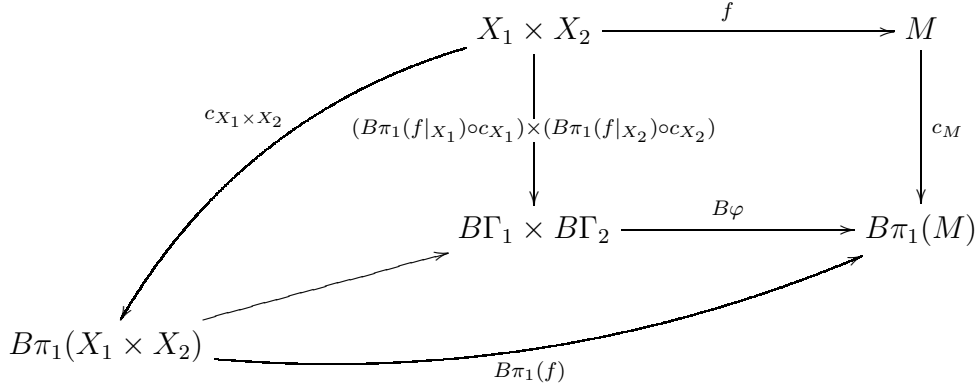
\begin{figure}
     \[
\xymatrix{
& & &X_1 \times X_2 \ar@/_2pc/[lllddd]_{c_{X_1 \times X_2}} \ar[dd]|-{(B\pi_1(f\vert_{X_1}) \circ c_{X_1}) \times (B\pi_1(f\vert_{X_2}) \circ
c_{X_2})} \ar[rrr]^{f} & & & M \ar[dd]^{c_M}\\
& & & & & & \\
& & &B\Gamma_1 \times B\Gamma_2 \ar[rrr]^{B\varphi} & & & B\pi_1(M)\\
B\pi_1(X_1 \times X_2) \ar[rrru]\ar@/_2pc/[rrrrrru]_{B\pi_1(f)}  & &}
     \]
\caption{\small{Domination by products on the level of classifying spaces.}}
\label{f:KotschickLoeh}
 \end{figure}

Suppose now that $M$ is a rationally essential $n$-dimensional manifold and let $f \colon X_1 \times X_2 \longrightarrow M$ be a map of non-zero degree,
where the $X_i$ are closed oriented manifolds of positive dimensions. Consider the induced map $\pi_1(f) \colon \pi_1(X_1) \times \pi_1(X_2)
\longrightarrow \pi_1(M)$ and set
\[
 \Gamma := \mathrm{im} (\pi_1(f)) \subset \pi_1(M) \ \textrm{ and } \ \Gamma_i := \mathrm{im}(\pi_1(f\vert_{X_i})) \subset \Gamma
\]
for the image of $\pi_1(f)$ and the images under $\pi_1(f)$ of the restrictions of $f$ to the two factors $X_i$ respectively. The multiplication map
$\varphi \colon
\Gamma_1 \times \Gamma_2 \longrightarrow \Gamma$ is then a well-defined surjective homomorphism, because the $\Gamma_i$ commute element-wise and $
\Gamma_1
\cup \Gamma_2$ generates $\Gamma$. Moreover, the outer commutative diagram in Figure \ref{f:KotschickLoeh} implies that $X_1 \times X_2$ is rationally
essential~\cite{KotschickLoeh1}.

Let $c_{X_i} \colon X_i \longrightarrow B\pi_1(X_i)$ be the classifying maps of the universal coverings of the $X_i$ and $B\pi_1(f\vert_{X_i}) \colon
B\pi_1(X_i) \longrightarrow B\Gamma_i$ be the maps induced by $\pi_1(f\vert_{X_i})$ on the level of classifying spaces. Also, let $B\varphi \colon
B\Gamma_1 \times B\Gamma_2 \longrightarrow B\Gamma$ be the map induced by $\varphi$; 
here we apply the homotopy
equivalence $B\Gamma_1 \times B\Gamma_2 \simeq B(\Gamma_1 \times \Gamma_2)$. We then have for $i = 1,2$ the
maps $ B\pi_1(f\vert_{X_i}) \circ c_{X_i} \colon X_i \longrightarrow B\Gamma_i$, and the corresponding rational homology classes
\begin{equation}
 \alpha_i := H_{\dim X_i}(B\pi_1(f\vert_{X_i})\circ c_{X_i})([X_i]) \in H_{\dim X_i}(B\Gamma_i;\Q),
\end{equation}
where $[X_i]$ denote the fundamental classes of the factors $X_i$.
According to this notation, the key observation of Kotschick and L\"oh~\cite{KotschickLoeh1}, shown in the commutative rectangle of Figure \ref{f:KotschickLoeh}, is 
\begin{eqnarray*}
 0 \neq \deg(f) \cdot H_n(c_M)([M]) = H_n(B\varphi)(\alpha_1 \times \alpha_2).
\end{eqnarray*}
This means that the $\alpha_i$ are not trivial and therefore
the $\Gamma_i$ are both infinite. In particular, $\Gamma$ is presented by the product $\varphi \colon \Gamma_1 \times \Gamma_2 \longrightarrow \Gamma$.
This proves the following:

\begin{thm}[Kotschick-L\"oh \cite{KotschickLoeh1}]\label{t:KotschickLoehmain}
 If $M$ is a rationally essential manifold and $\pi_1(M)$ is not presentable by products, then $M$ is not dominated by products.
 \end{thm}

\begin{rem}\label{r:KLGromov}
A consequence of Theorem \ref{t:KotschickLoehmain} is that Gromov's prediction was indeed correct. Namely, a locally symmetric space of non-compact type
is dominated by a product if and only
if it is virtually (isometric to) a product; cf.~\cite[Cor. 4.2]{KotschickLoeh1}. 
\end{rem}

\subsection{Groups presentable by products}\label{s:preliminariesPP}

We end this preliminary section with some elementary properties of groups presentable by products, mainly as introduced in~\cite[Section 3]{KotschickLoeh1}.

If a group $\Gamma$ is presentable by a product through a homomorphism $\varphi \colon \Gamma_1 \times \Gamma_2 \longrightarrow \Gamma$, then the images
$\varphi(\Gamma_i)$ commute with each other and $\varphi(\Gamma_1) \cup \varphi(\Gamma_2)$ generates $\mathrm{im}(\varphi)$. This means that, whenever a
group $\Gamma$ is presented by a product $\varphi \colon \Gamma_1 \times \Gamma_2 \longrightarrow \Gamma$, 
we can replace each $\Gamma_i$ by its image
$\varphi(\Gamma_i)$, $\Gamma$ by its finite index subgroup $\mathrm{im}(\varphi)$ and $\varphi$ by the multiplication map. Therefore we may always assume that $\Gamma$ can be presented by two element-wise commuting subgroups
$\Gamma_i$ through the multiplication map. The following properties can be easily verified:

\begin{lem}[\normalfont{\cite[Lemma 3.3]{KotschickLoeh1}}]\label{l:KotschickLoehproperiesPP}
Suppose $\Gamma_1, \Gamma_2$ are element-wise commuting subgroups of $\Gamma$ so that $\Gamma_1 \cup \Gamma_2$ generates $\Gamma$. Then the
multiplication map $\varphi \colon \Gamma_1 \times \Gamma_2 \longrightarrow \Gamma$ is a well-defined surjective homomorphism and the following statements
hold:
\begin{itemize}
\item[\normalfont{(1)}] the intersection $\Gamma_1 \cap \Gamma_2$ is a subgroup of the center $C(\Gamma)$;
\item[\normalfont{(2)}] the kernel of $\varphi$ is isomorphic to the Abelian group $\Gamma_1 \cap \Gamma_2$.
\end{itemize}
\end{lem}

In particular, there exists a short exact sequence
\begin{equation}\label{eq:PPses}
 1 \longrightarrow \Gamma_1\cap\Gamma_2 \longrightarrow \Gamma_1 \times \Gamma_2 \stackrel{\varphi}\longrightarrow \Gamma \longrightarrow 1,
\end{equation}
where the isomorphism between $\Gamma_1\cap\Gamma_2$ and the kernel of $\varphi$ is given by the antidiagonal.
For groups with finitely generated center we moreover observe the following:

\begin{lem}\label{l:finitelygeneratedcenter}
Let $\Gamma$ be a finitely generated group with finitely generated center. Assume that $\Gamma$ is 
presented by a product $\Gamma_1\times \Gamma_2$ as in Lemma \ref{l:KotschickLoehproperiesPP}. Then each of the factors $\Gamma_i$ is finitely generated.
\end{lem}
\begin{proof}
For $i,j \in \{1,2\}$, $i \neq j$, there exist (two) short exact sequences
\begin{equation}\label{eq:PPprojections}
 1 \longrightarrow \Gamma_i \longrightarrow \Gamma \longrightarrow \Gamma_j/(\Gamma_1 \cap \Gamma_2) \longrightarrow 1,
\end{equation}
where $\Gamma \longrightarrow \Gamma_j/(\Gamma_1 \cap \Gamma_2)$ is obtained by
composing the isomorphism $\Gamma \cong (\Gamma_1 \times \Gamma_2)/(\Gamma_1 \cap \Gamma_2)$ (cf. sequence (\ref{eq:PPses})) with the homomorphism induced
by the
projection from $\Gamma_1 \times \Gamma_2$ to $\Gamma_j$ (see also~\cite{KotschickLoeh2}). Since $\Gamma$ is finitely generated, the short exact sequence (\ref{eq:PPprojections}) implies that the group $\Gamma_j/(\Gamma_1 \cap \Gamma_2)$
is also finitely generated. Moreover, the center $C(\Gamma)$ is finitely generated Abelian and thus the intersection $\Gamma_1 \cap \Gamma_2$ is also
finitely generated Abelian by Lemma \ref{l:KotschickLoehproperiesPP}. This shows that $\Gamma_j$ is also finitely generated.
\end{proof}

\section{The case of reducible groups (proof of Theorem \ref{thmA})}

If $M$ is a closed aspherical $n$-dimensional manifold whose fundamental group $\pi_1(M)$ is reducible, then there exists a finite cover of $M$ with fundamental group isomorphic to a direct product $\Gamma_1\times\Gamma_2$. 
Thus, up to finite covers, we may identify $\pi_1(M)$ with $\Gamma_1\times \Gamma_2$. 
Then $B\Gamma_1\times B\Gamma_2$ is homotopy equivalent to $M$.  
In particular, 
there exists a non-trivial class $\alpha\in H_n(B\Gamma_1\times\ B\Gamma_2)$ mapping to the fundamental class $[M]\in H_n(M)$. 

Since each of the $\Gamma_i$ has infinite index in $\pi_1(M)$ and $M$ is aspherical manifold of dimension $n$, a theorem of Strebel~\cite{Sterbel} implies that the cohomological dimensions of each of the $\Gamma_i$ is less than $n$. Thus, the K\"unneth formula (with rational coefficients) in degree $n$ for the product $B\Gamma_1\times B\Gamma_2$ implies that there exist non-trivial homology classes $a_1\in H_k(B\Gamma_1)$ and $a_2\in H_{n-k}(B\Gamma_2)$, where $0<k<n$, such that $\alpha=a_1\otimes a_2$. 

By Thom's theorem \cite{Thom}, there exist two closed smooth manifolds $X_1$ and $X_2$ of dimensions $k$ and $n-k$ respectively, together with continuous maps $g_i\colon X_i\longrightarrow B\Gamma_i$, $i=1,2$, such that $H_*(g_i)([X_i])=d_i\cdot a_i$, for some non-zero integers $d_i$. Finally, the product map
\[
 X_1\times X_2\stackrel{g_1\times g_2}\longrightarrow B\Gamma_1\times B\Gamma_2\simeq M
\]
is continuous and in homology of degree $n$ maps the fundamental class $[X_1\times X_2]$ to a non-zero multiple of $[M]$. This finishes the proof of Theorem \ref{thmA}.

\section{Circle bundles with fundamental groups not IIPP \\(proofs of Theorems \ref{thmB} and \ref{thmC})}\label{s:IIPPdefinition} 

The purpose of this section is to introduce (in more detail) the property IIPP and prove Theorems \ref{thmB} and \ref{thmC}.

\subsection{Motivation and definition of the property IIPP}

Two basic examples of groups presentable by products are given by the reducible ones and by groups
containing a finite index subgroup with infinite center; cf. Example \ref{ex:examplesPPtf}. If follows by Lemma \ref{l:KotschickLoehproperiesPP} that
these two - not generally distinct - classes contain all torsion-free groups presentable by products \cite[Prop. 3.2]{KotschickLoeh1}.

A reducible group $\Gamma$ can always be presented (being a virtual product) by a product $\Gamma_1 \times \Gamma_2$ so that both subgroups $\Gamma_i$  have infinite index in $\Gamma$, whereas a group with infinite center does not generally have this property; a trivial example is given by the infinite cyclic group. On the topological level, Theorem \ref{t:KotschickLoehmain} states that whenever a rationally essential manifold $M$ is dominated by a product, its fundamental group is presentable by products. However, (the proof of) that result does not give any additional information on the index of the factors of a product presenting $\pi_1(M)$. The following example shows that all the possibilities can actually occur:


\begin{ex}\label{ex:allcasesindex} \
\begin{itemize}
 \item[\normalfont{(1)}] Let $M$ be a closed oriented manifold of positive dimension and infinite fundamental group. The identity map $\mathrm{id}_{M \times M}$ of
the product $M \times M$ is obviously $\pi_1$-surjective of degree one and both subgroups $\mathrm{im}(\pi_1(\mathrm{id}_{M})) = \pi_1(M)$ have infinite index in $\pi_1(M
\times M)$.

\item[\normalfont{(2)}] For $g \geq 1$, let $\Sigma_{g+1} = \Sigma_g \# (S_a^1 \times S_b^1)$ be a closed oriented surface of genus $g + 1$. Let
the composition
\begin{equation}\label{eq.counterex3}
 \Sigma_g \# (S_a^1 \times S_b^1) \stackrel{q}\longrightarrow  \Sigma_g \vee (S_a^1 \times S_b^1) \stackrel{\mathrm{id} \vee p}\longrightarrow \Sigma_g \vee S_b^1,
\end{equation}
where $q$ is the quotient map pinching to a point the essential circle defining the connected sum $\Sigma_g \# (S_a^1 \times S_b^1)$, $\mathrm{id}$ is
the identity map of $\Sigma_g$ and the map $p$ pinches to a point the meridian of the torus $S_a^1 \times S_b^1$; cf. Figure~\ref{f:counterex3}.
Denote by $h$ the composition $(\mathrm{id} \vee p) \circ q$. Now, let the composite map
\[
 \Sigma_{g+1} \times S_c^1 \stackrel{h \times \mathrm{id}_c}\longrightarrow (\Sigma_g \vee S_b^1) \times S_c^1 \stackrel{g}\longrightarrow \Sigma_g \times S_c^1,
\]
where $\mathrm{id}_c$ is the identity map of $S_c^1$ and $g$ is the identity on $\Sigma_g$ and $S_c^1$, and sends the generator $b$ of $S_b^1$ to the generator $c$ of $S_c^1$. 
\begin{figure}
\labellist
\pinlabel {$\stackrel{\mathrm{id} \vee p}\longrightarrow$} at 592 37
\pinlabel {$\stackrel{q}\longrightarrow$} at 287 37
\endlabellist
\centering
\includegraphics[width=15.9cm]{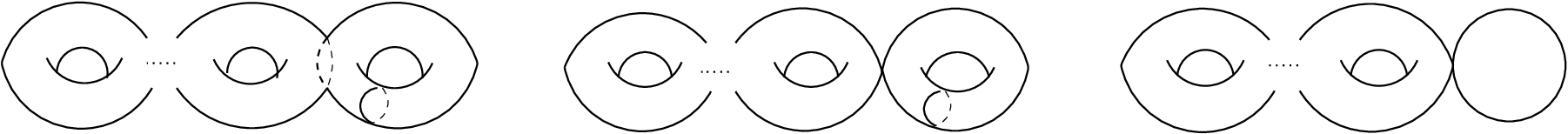}
\caption{\small The map $(\mathrm{id} \vee p) \circ q \colon \Sigma_{g+1} \longrightarrow \Sigma_g \vee S_b^1$ }
\label{f:counterex3}
\end{figure}
Let $f:= g \circ (h \times \mathrm{id}_c) \colon \Sigma_{g+1} \times S_c^1 \longrightarrow \Sigma_g \times S_c^1$. Then 
$
 H_3(f)([\Sigma_{g+1} \times S^1]) = [\Sigma_{g} \times S^1],
$
i.e. $\deg(f) = 1$. By the definition of $f$, we obtain an index-one subgroup of $\pi_1(\Sigma_g \times S^1)$, namely $\mathrm{im}(\pi_1(f\vert_{\Sigma_{g+1}})) =
\pi_1(\Sigma_g \times S^1)$, and the infinite-index subgroup $\mathrm{im}(\pi_1(f\vert_{S^1})) = \pi_1(S^1) \subset \pi_1(\Sigma_g \times S^1)$.

\item[\normalfont{(3)}] Let two copies of a closed oriented surface of genus three,
\begin{eqnarray*}
\Sigma_3 & = & (S_{g_1}^1 \times S_{g_2}^1) \# (S_{a_1}^1 \times S_{a_2}^1) \# (S_{b_1}^1 \times S_{b_2}^1) \\
\Sigma'_3 & = & (S_{g'_1}^1 \times S_{g'_2}^1)\# (S_{a'_1}^1 \times S_{a'_2}^1) \# (S_{b'_1}^1 \times S_{b'_2}^1).
\end{eqnarray*}
As in the previous example (cf. Figure \ref{f:counterex3}), define $h \colon \Sigma_3 \longrightarrow (S_{g_1}^1 \times S_{g_2}^1) \vee S_{a_2}^1 \vee
S_{b_2}^1$ as the composition
\[
 \Sigma_3 \stackrel{q}\longrightarrow (S_{g_1}^1 \times S_{g_2}^1) \vee (S_{a_1}^1 \times S_{a_2}^1) \vee (S_{b_1}^1 \times S_{b_2}^1) 
\stackrel{\mathrm{id} \vee p \vee p }\longrightarrow (S_{g_1}^1 \times S_{g_2}^1) \vee S_{a_2}^1 \vee S_{b_2}^1
\]
(see above for the notation). Now, let the composition
\[
 \Sigma_3 \times \Sigma'_3 \stackrel{h \times h'}\longrightarrow ((S_{g_1}^1 \times S_{g_2}^1) \vee S_{a_2}^1 \vee S_{b_2}^1) \times 
((S_{g'_1}^1 \times S_{g'_2}^1) \vee S_{a'_2}^1 \vee S_{b'_2}^1) 
\stackrel{g}\longrightarrow S_{g_1}^1 \times S_{g_2}^1 \times S_{g'_1}^1 \times S_{g'_2}^1,
\]
where $h,h'$ are defined above, and $g$ restricts to the identity map on each $S_{g_j}^1$ and $S_{g'_j}^1$, and is given as follows on the rest of
the circles:
\begin{eqnarray*}
 a_2 \mapsto g'_1, \ \ b_2 \mapsto g'_2, \ \ a'_2 \mapsto g_1, \ \ b'_2 \mapsto g_2.
\end{eqnarray*}

We define $f \colon \Sigma_3 \times \Sigma'_3 \longrightarrow T^4$ to be the composition $g \circ (h \times h')$. Again, $f$ is a degree one map. However,
both $\mathrm{im}(\pi_1(f\vert_{\Sigma_3}))$ and $\mathrm{im}(\pi_1(f\vert_{\Sigma'_3}))$ are now of index one in $\pi_1(T^4)$. 

We note that this construction cannot be generalized when the target is not a product of two tori, $T^2 \times T^2$, because the generators of higher genus
surfaces do not commute with each other. Actually, it will be transparent by the discussion in the upcoming subsection (cf. Lemma \ref{l:notIIPPalgebraic}),
that, if an $n$-dimensional aspherical
manifold $M$ admits a map $f \colon X_1 \times X_2 \longrightarrow M$ so that both subgroups $\mathrm{im}(\pi_1(f\vert_{X_i})) \subset \pi_1(M)$ are of finite
index, then $M$ is a virtual $n$-dimensional torus $T^n$.
\end{itemize}
\end{ex}

In this paper, we analyze groups presentable by products by adding a constraint on the index of the presenting factors. More precisely, we introduce the following class of groups: 

\begin{defn}\label{d:IIPP}
 An infinite group $\Gamma$ is called {\em infinite-index presentable by products} (IIPP) if there is a homomorphism $\varphi \colon \Gamma_1 \times
\Gamma_2 \longrightarrow \Gamma$ onto a finite index subgroup of $\Gamma$ so that for both factors $\Gamma_i$ the images $\varphi(\Gamma_i) \subset \Gamma$
are of infinite index in $\Gamma$. 
Otherwise, $\Gamma$ is called {\em not infinite-index presentable by products} (not IIPP).
\end{defn}

In the upcoming subsection, we will show that an aspherical manifold with fundamental group not IIPP can be dominated by a product only if it is dominated by a product containing a torus factor. This will imply that large classes of aspherical manifolds with fundamental groups not IIPP cannot be dominated by products. In the case of circle bundles, we will prove that, under a certain assumption on the fundamental group of the base, the condition ``IIPP'' characterizes aspherical circle bundles that are dominated by products. Without that additional assumption on the base, this characterization does not generally hold (as we have already mentioned in Example \ref{ex:Heisenberg}).

\subsection{Not IIPP as a non-domination criterion (proof of Theorem \ref{thmB})}

We first extend the non-existence results of~\cite{KotschickLoeh1} to certain rationally essential manifolds with fundamental groups presentable by products,
but not IIPP.  The strong feature of such torsion-free groups is that one of the presenting subgroups must be Abelian: 

\begin{lem}\label{l:notIIPPalgebraic}
Let $\Gamma$ be a finitely generated torsion-free group with finitely generated center. Suppose that there exist element-wise commuting subgroups $\Gamma_1, \Gamma_2
\subset \Gamma$ so that $\Gamma_1 \cup \Gamma_2$ generates $\Gamma$. If $\Gamma$ is not IIPP, then one of the $\Gamma_i$ is isomorphic to
$\Z^k$ for some $k \leq \mathrm{rank} C(\Gamma)$.
\end{lem}
\begin{proof}
By Lemma \ref{l:KotschickLoehproperiesPP}, there is a short exact sequence
\[
 1 \longrightarrow \Gamma_1\cap\Gamma_2 \longrightarrow \Gamma_1 \times \Gamma_2 \stackrel{\varphi}\longrightarrow \Gamma \longrightarrow 1,
\]
where $\varphi$ is the multiplication map and the intersection $\Gamma_1 \cap \Gamma_2$ is contained in the finitely generated center $C(\Gamma)$. Since
$\Gamma$ is torsion-free, we have that $\Gamma_1 \cap \Gamma_2$ is isomorphic to $\Z^k$ for some $k \leq \mathrm{rank} C(\Gamma)$. Moreover, each $\Gamma_i$ is
finitely generated by Lemma \ref{l:finitelygeneratedcenter}.

Because $\Gamma$ is not IIPP, one of the $\Gamma_i$, say $\Gamma_1$, must have finite index in $\Gamma$. This means that $\Gamma_2$ is virtually
$\Gamma_1 \cap \Gamma_2$, and so it is virtually Abelian. Moreover, the intersection $\Gamma_1 \cap \Gamma_2$ is central in $\Gamma_2$, which implies that
$\Gamma_2/C(\Gamma_2)$ is finite (because $\Gamma_2$ is virtually $\Gamma_1 \cap \Gamma_2$). By Schur's theorem~\cite{Schur}, we conclude that the commutator
$[\Gamma_2,\Gamma_2]$ is also finite and so trivial, because $\Gamma_2$ is torsion-free. This shows that $\Gamma_2$ is Abelian itself and thus
isomorphic to $\Z^k$.
\end{proof}

As a warm-up, we observe that, in the torsion-free case, Lemma \ref{l:notIIPPalgebraic} yields the following dimension restrictions on the factors of a product that dominates a rationally essential manifold:

\begin{prop}\label{p:notIIPPdimensions}
 Let $M$ be a rationally essential manifold so that $\pi_1(M)$ is torsion-free and $\mathrm{rank} C(\pi_1(M)) = r$. If $\pi_1(M)$ is not IIPP,
then there is no $\pi_1$-surjective non-zero degree map $X_1 \times X_2 \longrightarrow M$, whenever $\min \{\dim X_1, \dim X_2\} > r$.
\end{prop}

\begin{proof}
 Suppose that there exist $X_1$, $X_2$ closed oriented manifolds of positive dimensions and a $\pi_1$-surjective non-zero degree map $f \colon
X_1 \times X_2 \longrightarrow M$. Then there is a short exact sequence
\begin{equation}\label{eq:sesnonzerodegree}
 1 \longrightarrow \Gamma_1\cap\Gamma_2 \longrightarrow \Gamma_1 \times \Gamma_2 \stackrel{\varphi}\longrightarrow \pi_1(M) \longrightarrow 1,
\end{equation}
where $\varphi$ is the multiplication map, $\Gamma_i := \pi_1(f\vert_{X_i})(\pi_1(X_i))$ and $\Gamma_1 \cap \Gamma_2 \subset C(\pi_1(M))$; see
Section \ref{s:notation}.
In particular, $\Gamma_1 \cap \Gamma_2$ is isomorphic to $\Z^k$, for some $k \leq r = \mathrm{rank} C(\pi_1(M))$, because it is torsion-free. We moreover observe that
$k \geq 1$,
otherwise $\pi_1(M)$ would be isomorphic to the product $\Gamma_1 \times \Gamma_2$ by (\ref{eq:sesnonzerodegree}) and so IIPP.

We now apply Lemma \ref{l:notIIPPalgebraic} to $\pi_1(M)$ to conclude that one of the $\Gamma_i$, say $\Gamma_2$, is isomorphic to $\Z^k$. This means
that $B\Gamma_2 \simeq T^k$ and by the non-vanishing (cf. Section \ref{s:notation}) of 
\[
 \alpha_2 := H_{\dim X_2}(B\pi_1(f\vert_{X_2})\circ c_{X_2})([X_2]) \in H_{\dim X_2}(T^k;\Q),
\]
we deduce that $\dim X_2 \leq k$. This is possible only if $\min \{\dim X_1, \dim X_2\} \leq r$.
\end{proof}

In the case where $M$ is aspherical, then the above proposition says that, if $f\colon X_1\times X_2\longrightarrow M$ is a $\pi_1$-surjective map of non-zero degree, then there is a finite cover $\overline{M}$ of $M$ such that $f$ factors through the map $B\varphi \colon \overline{M}\times T^k\longrightarrow M$ which is induced by the multiplication homomorphism $\varphi\colon\Gamma_1\times\Gamma_2\longrightarrow\pi_1(M)$, for some $k\leq\mathrm{rank} C(\pi_1(M))$; cf.  Figure \ref{f:KotschickLoeh}. In particular, there exist two non-trivial homology classes $\alpha_1\in H_{\dim X_1}(\overline{M};\Q)$ and $\alpha_2\in H_{\dim X_2}(T^k;\Q)$ such that $H_n(B\varphi)(\alpha_1\times\alpha_2)=\deg(f)\cdot[M]$. Thus $X_1\times T^m\geq M$, where $m=\dim X_2\leq k$.

\begin{ex}[$1$-domination]
 Let $M$ be a closed aspherical manifold. If $\pi_1(M)$ has infinite cyclic center and it is not IIPP, then $M$ can admit a
degree one map by a product $X_1 \times X_2$ only if one of the $X_i$ is a circle. (Recall that a map of degree one is $\pi_1$-surjective.)
\end{ex}

Aspherical manifolds whose fundamental groups have infinite cyclic center are of special interest, 
in particular with respect to the study of circle bundles and circle actions; see~\cite{CWY} and the references there. We begin with two general facts about finite coverings of circle bundles:

\begin{lem}\label{l:propertiescirclebundles}
Let $M \stackrel{\pi}\longrightarrow B$ be a circle bundle over a closed oriented manifold $B$.
\begin{itemize}
 \item[\normalfont{(1)}] Every finite cover $\overline{M} \stackrel{p}\longrightarrow M$ is a circle bundle over a finite
cover $B' \stackrel{p'}\longrightarrow B$. If moreover $B$ is aspherical and $\pi_1(B)$ is not presentable by products, then $\pi_1(\overline{M})$ and
$\pi_1(M)$ have infinite cyclic center.
 \item[\normalfont{(2)}] {\normalfont{(\cite{Bowden})}} If the Euler class of $M$ is torsion, then $M$ is a virtually trivial circle bundle over a
finite cover of $B$.
\end{itemize}
\end{lem}
\begin{proof}
 (1) Since $\pi_1(p)(\pi_1(\overline{M}))$ has finite index in $\pi_1(M)$ and $\pi_1(\pi)(\pi_1(M))=\pi_1(B)$, the image
$
 H:=\pi_1(\pi \circ p)(\pi_1(\overline{M}))
$
has finite index in $\pi_1(B)$. Let $B' \stackrel{p'}\longrightarrow B$ be the finite covering corresponding to $H$. Then $\pi \circ p$ lifts to
$\overline{M} \stackrel{\pi'}\longrightarrow B'$, which is the desired circle bundle.

If $B$ is aspherical, then the $S^1$ fiber is central in the fundamental group of $M$. If, in addition, $\pi_1(B)$ is not presentable by products, then it
has
trivial center (because it is torsion-free), and so the center of $\pi_1(M)$ is infinite cyclic. Now $\pi_1(B')$ has finite index
in $\pi_1(B)$, and so it is not presentable by products as well and therefore the center of
$\pi_1(\overline{M})$ is also infinite cyclic.  

(2) Consider the Abelianization $H_1(B) = \pi_1(B)/[\pi_1(B),\pi_1(B)]$. Since the Euler class of $M$ is torsion, the Universal Coefficient Theorem
implies that the torsion part of $H_1(B)$ is not trivial. Let now the composition
\[
 \pi_1(B) \longrightarrow H_1(B) \longrightarrow \mathrm{Tor}H_1(B),
\]
where the first map is the quotient map and the second is the projection to the torsion of $H_1(B)$. If $B' \stackrel{p'}\longrightarrow B$ is the finite
covering corresponding to the kernel of the above composition, then the pullback bundle $(p') ^*(M)$ is the product $S^1 \times B'$;
see~\cite[Prop. 3]{Bowden} for more details.
\end{proof}
 
\begin{rem}\label{r:Torsion-freeEuler}
 Conversely to part (2) of the above lemma, let $\overline{M} = S^1 \times B' \stackrel{p}\longrightarrow M$ be a finite cover, where
$B' \stackrel{p'}\longrightarrow B$ is a finite covering between the bases (the map $p'$ is covered by $p$). The Euler class of $\overline{M}$ is
trivial, that is
$e_{\overline{M}}=H^2(p';\Z)(e_M)=0 \in H^2(B',\Z)$, where $e_M \in H^2(B;\Z)$ is the Euler class of $M$. By the fact that $H^2(p';\Q)$ is injective (since $\deg(p')\neq0$), we conclude that $e_M$ is torsion.
\end{rem}

A basic ingredient of our proof is the following lemma which generalizes ~\cite[Lemma 1]{KotschickNeofytidis}:

\begin{lem}[Factorization Lemma]\label{l:KotschickNeofytidisgeneralization}
 Let $M \stackrel{\pi}\longrightarrow B$ be a non-trivial $n$-dimensional circle bundle over a closed oriented aspherical manifold $B$. Suppose
that the Euler class of $M$ is not torsion and that the center of $\pi_1(M)$ remains infinite cyclic in finite covers. Then $X \times S^1 \ngeq M$ for any closed oriented
manifold $X$.
\end{lem}
\begin{proof}
Since $M$ is a non-trivial circle bundle whose integer Euler class $e_M \in H^2(B;\Z)$ is not torsion, the rational Euler class of $M$ is not trivial as
well. The same property holds for every (fiber preserving) finite cover of $M$, by Lemma~\ref{l:propertiescirclebundles} and Remark
\ref{r:Torsion-freeEuler}. By Poincar\'e duality, there exists a
non-trivial class $\alpha \in H^{n-3}(B;\Q)$ so that $e_M \smile \alpha$ is a non-zero multiple of the cohomology fundamental class $\omega_B$ of $B$. Since $H^{n-1}(B;\Q) = \Q$, the Gysin sequence
\[
 \cdots \longrightarrow H^{n-3}(B;\Q) \stackrel{\smile e_M}\longrightarrow H^{n-1}(B;\Q) \stackrel{H^{n-1}(\pi)} \longrightarrow H^{n-1}(M;\Q)
\longrightarrow \cdots
\]
implies that $\ker (H^{n-1}(\pi)) = \mathrm{im} (\smile e_M) = H^{n-1}(B;\Q)$. Therefore $H^{n-1}(\pi) = 0$.

Suppose now that there exists a non-zero degree map $f \colon X \times S^1 \longrightarrow M$. After passing to a finite cover, if necessary, we may assume
that $f$ is $\pi_1$-surjective and that the center of $\pi_1(M)$ is infinite cyclic. 
The latter means that the circle fiber of $M$ represents (up to multiples) the
only central factor in $\pi_1(M)$. By the surjectivity of $\pi_1(f)$, we deduce that the composite map $\pi \circ f$ kills the homotopy class of the
$S^1$ factor
of the product $X \times S^1$, because this factor is central in $\pi_1(X \times S^1)$. Since $B$ is aspherical, we conclude that $\pi \circ f$
factors up to homotopy through the projection $p_1 \colon X \times S^1 \longrightarrow X$. In particular, there is a continuous map $g \colon X
\longrightarrow B$, so that $\pi \circ f = g \circ p_1$ up to homotopy. (We note that $X$ is not necessarily aspherical. It is, however, rationally
essential, because $f$ has non-zero degree and $M$ is aspherical.)

Let $\omega_X$ be the cohomology fundamental class of $X$. Since 
$H^{n-1}(p_1;\Q)(\omega_X) = \omega_X \in H^{n-1}(X \times S^1;\Q)$ and
$H^{n-1}(\pi;\Q)(\omega_B) = 0 \in H^{n-1}(M;\Q)$,
the homotopy equation $\pi \circ f = g \circ p_1$ implies that $g$ must be of zero degree. Let now
the pullback of $M$ under $g$:
\[
g^*M = \{ (x,y)\in X \times M \ \vert \ g(x) = \pi (y) \} \ .
\]
The map $f \colon X \times S^1 \longrightarrow M$ factors through $g^*M$ as follows:
\begin{eqnarray*}
 X \times S^1 \longrightarrow & g^*M       & \stackrel{\pi_2}{\longrightarrow} M\\
(x,t)                 \mapsto     & (x,f(x,t)) & \mapsto f(x,t) \ .
\end{eqnarray*}   
However, the degree of the pullback map $\pi_2 \colon  g^*M \longrightarrow M$ is zero, being equal to the degree of $g$, which contradicts our assumption on $\deg(f)$. This completes the proof.
\end{proof}

We now finish the proof of Theorem \ref{thmB}:

\begin{proof}[Proof of Theorem \ref{thmB}]
Since $\pi_1(M)$ is not IIPP, $M$ is a non-trivial circle bundle and, moreover, its Euler class is not torsion by Lemma \ref{l:propertiescirclebundles}
(2).
After passing to a finite cover, if necessary, suppose that there is a $\pi_1$-surjective non-zero degree map $f \colon X_1 \times X_2 \longrightarrow
M$, where $\dim X_i > 0$ and $C(\pi_1(M)) =\Z$; cf. Lemma \ref{l:propertiescirclebundles} (1). As before, there
is a short exact sequence
\[
 1 \longrightarrow \Gamma_1\cap\Gamma_2 \longrightarrow \Gamma_1 \times \Gamma_2 \stackrel{\varphi}\longrightarrow \pi_1(M) \longrightarrow 1,
\]
where $\Gamma_i := \mathrm{im}(\pi_1(f\vert_{X_i})) \subset \pi_1(M)$, $\varphi$ is the multiplication map and $\Gamma_1\cap\Gamma_2 \subset
C(\pi_1(M)) = \Z$, see Section \ref{s:notation}. 

Lemma \ref{l:notIIPPalgebraic} implies that one of the $\Gamma_i$, say $\Gamma_2$, must be infinite cyclic,
because $\pi_1(M)$ is not IIPP and torsion-free. Therefore, $B\Gamma_2 \simeq S^1$ and because the rational class
\[
 \alpha_2 := H_{\dim X_2}(B\pi_1(f\vert_{X_2})\circ c_{X_2})([X_2]) \in H_{\dim X_2}(S^1;\Q)
\]
is not trivial we conclude that $\dim X_2 = 1$, i.e. $X_2 = S^1$. Now, we have a $\pi_1$-surjective dominant map $X_1 \times S^1
\longrightarrow M$, where $C(\pi_1(M)) = \Z$. The proof follows by Lemma \ref{l:KotschickNeofytidisgeneralization}.
\end{proof}

\subsection{A characterization for circle bundles (proof of Theorem \ref{thmC})}\label{s:characterizationcb}

A main motivation for Theorem \ref{t:KotschickLoehmain} was to show that non-positively curved manifolds
which are not virtual products are not dominated by products. Actually, the property ``fundamental group presentable by products'' suffices for
domination by products for non-positively curved manifolds (of dimension higher than one) and it is equivalent to ``reducible"; cf. \cite[Theorem 4.1]{KotschickLoeh1}. Another consequence of the results of~\cite{KotschickLoeh1}, which moreover deals with manifolds that do not admit any metric of non-positive
sectional curvature (cf.~\cite{KapLeeb}), concerns fibrations whose fiber and base have fundamental groups not presentable by products:

\begin{thm}[\normalfont{\cite[Theorem 5.1]{KotschickLoeh1}}]\label{t:fiberbundles}
 Let $F \longrightarrow M \stackrel{\pi}\longrightarrow B$ be a fiber bundle whose fiber $F$ and base $B$ are closed oriented aspherical manifolds with
fundamental groups not presentable by products. Then $M$ is dominated by products if and only if it is a virtual product $F' \times B'$, where $F'$ and
$B'$ are finite covers of $F$ and $B$ respectively.
\end{thm}


\begin{cor}[\normalfont{\cite[Cor. 5.3]{KotschickLoeh1}}]\label{c:KLsurfaceb}
 Let $M$ be a closed oriented $4$-manifold which is the total space of a surface bundle whose fiber $F$ and base $B$ are both hyperbolic surfaces. Then
the following are equivalent:
\begin{itemize}
 \item[\normalfont{(1)}] $M$ is dominated by a non-trivial product of closed oriented manifolds;
 \item[\normalfont{(2)}] $M$ is virtually diffeomorphic to a trivial surface bundle;
 \item[\normalfont{(3)}] $\pi_1(M)$ is reducible;
 \item[\normalfont{(4)}] $\pi_1(M)$ is presentable by products.
\end{itemize}
\end{cor}

\begin{rem}
The fact that hyperbolic groups are not presentable by products is proved in~\cite{KotschickLoeh1}. Moreover, we note that $4$-manifolds satisfying one (and therefore every) property in the above corollary constitute the class of closed reducible $\mathbb{H}^2 \times \mathbb{H}^2$-manifolds; cf. Section \ref{ss:enumeration4-mfds}.
\end{rem}

Now, if we replace the fiber $F$ by $S^1$ in Theorem \ref{t:fiberbundles}, then $\pi_1(M)$ is presentable by products having infinite cyclic center. Theorem \ref{thmC}, which is a partial converse of Theorem \ref{thmB}, says that the conclusion of Theorem \ref{t:fiberbundles} still holds when the fiber is $S^1$. However, domination by products is now equivalent to the conditions ``$\pi_1(M)$ IIPP'' and ``$\pi_1(M)$ reducible''; compare with the equivalence between $(3)$ and $(4)$ of Corollary \ref{c:KLsurfaceb}. 

\begin{proof}[Proof of Theorem \ref{thmC}]
 Since $B$ is aspherical, $M$ is also aspherical and its fundamental group fits into a short exact sequence
\begin{equation}
 1 \longrightarrow \pi_1(S^1) \longrightarrow \pi_1(M) \stackrel{\pi_1(\pi)}\longrightarrow \pi_1(B) \longrightarrow 1,
\end{equation}
where $\pi_1(S^1)$ is in the center of $\pi_1(M)$. Moreover, $\pi_1(B)$ has trivial center, because it is torsion-free and not presentable by products.
Thus $C(\pi_1(M)) = \pi_1(S^1) = \Z$.

Suppose that there is a non-zero degree map $f \colon X_1 \times X_2 \longrightarrow M$. After passing to a finite cover, if necessary, we may
assume that $f$ is $\pi_1$-surjective. (The finite cover of $M$ is a circle bundle with infinite cyclic center, by Lemma
\ref{l:propertiescirclebundles} (1).) As before, we have a short exact sequence
\begin{eqnarray*}
 1 \longrightarrow \Gamma_1\cap\Gamma_2 \longrightarrow \Gamma_1 \times \Gamma_2 \stackrel{\varphi}\longrightarrow \pi_1(M) \longrightarrow 1,
\end{eqnarray*}
where  $\Gamma_i := \mathrm{im}(\pi_1(f\vert_{X_i})) \subset \pi_1(M)$ and $\Gamma_1\cap\Gamma_2 \subset C(\pi_1(M)) = \Z$.
Moreover, we obtain two non-trivial rational homology classes
\begin{eqnarray*}
 \alpha_i := H_{\dim X_i}(B\pi_1(f\vert_{X_i})\circ c_{X_i})([X_i]) \neq 0 \in H_{\dim X_i}(B\Gamma_i;\Q),
\end{eqnarray*}
see Sections \ref{s:notation} and \ref{s:preliminariesPP} for the details. 

The composite homomorphism 
\begin{eqnarray*}
\Gamma_1 \times \Gamma_2 \stackrel{\varphi}\longrightarrow \pi_1(M) \stackrel{\pi_1(\pi)}\longrightarrow \pi_1(B) \cong \pi_1(M)/\pi_1(S^1)
\end{eqnarray*}
maps one of the $\Gamma_i$, say $\Gamma_1$, to the neutral element of $\pi_1(B)$, because $\pi_1(B)$ is not presentable by products and torsion-free. This
means that $\Gamma_1$ is contained in $C(\pi_1(M)) = \pi_1(S^1) = \Z$ and it is therefore isomorphic to $C(\pi_1(M)) = \Z$. In
particular, $B\Gamma_1 \simeq B\Z =
S^1$ and so the non-vanishing of $\alpha_1 \in H_{\dim X_1}(S^1;\Q)$
implies that $\dim X_1 \leq 1$. Since $\dim X_1 > 0$, we have that $X_1 = S^1$, i.e. $S^1 \times X_2 \geq M$. It follows by
Lemma \ref{l:KotschickNeofytidisgeneralization} that $M$ is a virtual product and, more precisely, that it is finitely covered by a product $S^1 \times B'$
for
some finite cover $B' \longrightarrow B$. Thus (1)
implies (2). The converse is trivially true and so (1) is equivalent to (2).

Next, we show that the properties (2) and (3) are equivalent. Obviously (2) implies (3). Assume now that $\pi_1(M)$ is reducible. Then there exists
a finite cover $M' \longrightarrow M$ so that $\pi_1(M')$ is isomorphic to a direct product $\Delta_1 \times \Delta_2$, where $\Delta_i$ are non-trivial (and therefore infinite)
subgroups of $\pi_1(M')$. The cover $M'$ is a circle bundle over a finite cover $B'$ of $B$, where $\pi_1(B')$ is not presentable by products being a
finite index subgroup of $\pi_1(B)$; cf. Lemma \ref{l:propertiescirclebundles} (1). We therefore obtain a short exact sequence
\begin{equation}\label{eq.vproductsequence}
 1 \longrightarrow \pi_1(S^1) \longrightarrow \Delta_1 \times \Delta_2 \longrightarrow \pi_1(B') \longrightarrow 1,
\end{equation}
where $\pi_1(S^1) = C(\pi_1(M')) \cong C(\Delta_1) \times C(\Delta_2)$. 
Since $\pi_1(B')$ is not presentable by products and torsion-free, one of the $\Delta_i$, say $\Delta_1$, maps trivially to
$\pi_1(B') \cong \pi_1(M')/\pi_1(S^1)$ in (\ref{eq.vproductsequence}). Thus $\Delta_1 \subset \pi_1(S^1)$ (and so $\Delta_1$ is isomorphic to $\Z$) and
$\Delta_2$ surjects onto $\pi_1(B')$. Moreover, $\pi_1(S^1)$ maps trivially to $\Delta_2$, otherwise $\Delta_2$ would have finite
index in $\pi_1(M')$, which is impossible, because $\pi_1(M') \cong \Delta_1 \times \Delta_2$ and both $\Delta_i$ are infinite. Therefore $\Delta_2$ maps isomorphically onto $\pi_1(B')$. We have now proved
that $\pi_1(M') \cong \pi_1(S^1 \times B')$ and so $M'$ is homotopy equivalent to $S^1 \times B'$. Thus (3) implies (2).

Finally, the equivalence between $(3)$ and $(4)$ follows from the more general group-theoretic Theorem \ref{thmD}, whose proof is given in the upcoming section.
\end{proof}

\begin{rem}
 An alternative argument for the last step in the proof of the implication $(3) \Rightarrow (2)$ is the following: 
 Having that $\Delta_1$ maps trivially to $\pi_1(B') \cong \pi_1(M')/\pi_1(S^1)$, we conclude that $\pi_1(S^1)$ maps trivially to $\Delta_2$ because the center of $\pi_1(M') \cong \Delta_1 \times \Delta_2$ is infinite cyclic, isomorphic to $\pi_1(S^1)$. Actually, taking for granted that the circle fiber of $M'$ is the only central factor in $\pi_1(M')$, we can relax the condition ``not presentable by products'' for the fundamental group of the base $B'$ to ``irreducible''.
 
Also, note that after showing the implication $(1)\Rightarrow (2)$ and since the implication $(2)\Rightarrow (3)$ is trivial, we can deduce the equivalence of $(1), (2)$ and $(3)$ by Theorem \ref{thmA} (which gives the implication $(3)\Rightarrow (1)$).
 \end{rem}

This discussion yields a topological example of groups not IIPP in any dimension:

\begin{cor}
 If $M$ is a circle bundle with non-trivial rational Euler class over a closed aspherical manifold $B$ so that $\pi_1(B)$ is not presentable by products, then $\pi_1(M)$ is not IIPP.
\end{cor}
\begin{proof}
Since $\pi_1(B)$ is not presentable by products, $\pi_1(M)$ is IIPP if and only if it is reducible, by the equivalence between $(3)$ and $(4)$ in Theorem
\ref{thmC} (or by Theorem \ref{thmD} below). However, $\pi_1(M)$ is not reducible, otherwise $M$ would be
covered
by $S^1 \times B'$, for some finite cover $B' \longrightarrow B$ (by the equivalence between $(2)$ and $(3)$ of Theorem
\ref{thmC}), which is impossible because the Euler class of $M$ is not torsion; cf. Remark \ref{r:Torsion-freeEuler}. 
\end{proof}

\section{The IIPP property as a criterion for reducibility (proof of Theorem \ref{thmD})}

In this section we show that the properties ``IIPP" and ``reducible" are equivalent for a group $\Gamma$, whenever $\Gamma/C(\Gamma)$ is not presentable by products.

\begin{proof}[Proof of Theorem \ref{thmD}]
Since every reducible group is IIPP, we only need to show that the converse is also true when the quotient of our group by its center is not presentable by products. 

Let $\Gamma$ be an IIPP group such that $\Gamma/C(\Gamma)$ is not presentable by products. Since $\Gamma$ is IIPP, there exists a finite index subgroup $\overline{\Gamma}\subset\Gamma$ and two element-wise commuting subgroups $\Gamma_1,\Gamma_2\subset\overline{\Gamma}$ of infinite index such that $\overline{\Gamma}=\Gamma_1\Gamma_2$. Thus we obtain the following composite surjective homomorphism
\[
\Gamma_1\times\Gamma_2 \longrightarrow \overline{\Gamma} \longrightarrow \overline{\Gamma}/\overline{\Gamma}\cap C(\Gamma),
\]
whose image $\overline{\Gamma}/\overline{\Gamma}\cap C(\Gamma)$ is of finite index in $\Gamma/C(\Gamma)$. In particular, $\overline{\Gamma}/\overline{\Gamma}\cap C(\Gamma)$ is not presentable by products. Thus, for one of the $\Gamma_i$, say for $\Gamma_2$, the quotient $\Gamma_2/\Gamma_2\cap C(\Gamma)$ is finite. Since $\Gamma_2\cap C(\Gamma)$ is central in $\Gamma_2$, we deduce that $\Gamma_2$ is virtually $C(\Gamma_2)$, in particular it is virtually Abelian. 

We claim that $[\overline{\Gamma}:\Gamma_1C(\Gamma_2)]<\infty$. First, we have that
\begin{eqnarray}\label{eq:Gammabar}
\overline{\Gamma}=\Gamma_1\Gamma_2=\Gamma_1C(\Gamma_2)\Gamma_2.
\end{eqnarray}
Next, we observe that 
\begin{eqnarray}\label{eq:product-center}
\Gamma_1C(\Gamma_2)\cap\Gamma_2=C(\Gamma_2).
\end{eqnarray}
Indeed, on the one hand it is clear that $C(\Gamma_2)\subset\Gamma_1C(\Gamma_2)\cap\Gamma_2$. On the other hand, every element in $\Gamma_1C(\Gamma_2)\cap\Gamma_2$ is central in $\Gamma_2$, because it belongs to $\Gamma_2$ and the $\Gamma_i$ commute element-wise. By (\ref{eq:Gammabar}), (\ref{eq:product-center}) and since $\Gamma_2/C(\Gamma_2)$ is finite, we obtain
\[
[\overline{\Gamma}:\Gamma_1C(\Gamma_2)]=[\Gamma_2:\Gamma_1C(\Gamma_2)\cap\Gamma_2]=[\Gamma_2:C(\Gamma_2)]<\infty,
\]
as claimed.

Let now the presentation of $\Gamma_1C(\Gamma_2)$ by
\begin{eqnarray}\label{eq:productpresentation}
\Gamma_1\cap C(\Gamma_2)\longrightarrow\Gamma_1\times C(\Gamma_2)\longrightarrow\Gamma_1 C(\Gamma_2).
\end{eqnarray}
Since $[\overline{\Gamma}:\Gamma_1]=\infty$ and $[\overline{\Gamma}:\Gamma_1C(\Gamma_2)]<\infty$, we conclude that 
\[
[C(\Gamma_2):\Gamma_1\cap C(\Gamma_2)]=[\Gamma_1C(\Gamma_2):\Gamma_1]=\infty.
\]
Thus $\Gamma_1C(\Gamma_2)$ is isomorphic to the product $\Gamma_1\times (C(\Gamma_2)/\Gamma_1\cap C(\Gamma_2))$, where the quotient group $C(\Gamma_2)/\Gamma_1\cap C(\Gamma_2)$ is Abelian of positive rank.
\end{proof}

Theorem \ref{thmD} proves in particular the equivalence between $(3)$ and $(4)$ in Theorem \ref{thmC}, for circle bundles over aspherical manifolds with fundamental groups not presentable by products.

As pointed out in the introduction, the property IIPP on a group $\Gamma$ is not anymore a criterion for reducibility of $\Gamma$ if the quotient $\Gamma/C(\Gamma)$ is presentable by products. More precisely, we have seen in Example \ref{ex:Heisenberg} that the $5$-dimensional Heisenberg group $H_5$ -- whose quotient $H_5/C(H_5)$ is $\Z^4$ -- is irreducible and not IIPP. This nilpotent group is realized as the fundamental group of a circle bundle over $T^4$. As we shall see in the next section (cf. Theorem \ref{thmE}), dimension four is the sharp dimension in which irreducible fundamental groups of solvable manifolds are not IIPP.

We note that the center of $\Gamma$ in Theorem \ref{thmD} can be assumed to be infinite, otherwise $\Gamma$ is trivially not presentable by products. At the other end, for groups whose every finite index subgroup has finite center, the notions ``presentable by products'' and ``IIPP'' are equivalent:

\begin{prop}
 If every subgroup of finite index in $\Gamma$ has finite center, then $\Gamma$ is presentable by products if and only if it is IIPP. 
\end{prop}
\begin{proof}
 It suffices to show that presentability by products implies IIPP. Suppose that $\Gamma_1$, $\Gamma_2$ are commuting infinite subgroups of $\Gamma$ and
that there is a short exact sequence
\[
 1 \longrightarrow \Gamma_1\cap\Gamma_2 \longrightarrow \Gamma_1 \times \Gamma_2 \stackrel{\varphi}\longrightarrow \Gamma \longrightarrow 1,
\]
where $\varphi$ is the multiplication homomorphism and $\Gamma_1 \cap \Gamma_2$ lies in the center of $\Gamma$; cf. Lemma \ref{l:KotschickLoehproperiesPP}.
Since $C(\Gamma)$ is finite, we deduce that $\Gamma_1 \cap \Gamma_2$ is also finite and so it has infinite index in both $\Gamma_i$. The proof now
follows by the short exact sequence (\ref{eq:PPprojections}) in the proof of Lemma \ref{l:finitelygeneratedcenter}.
\end{proof}

\section{Fundamental groups of geometric manifolds in low dimensions \\(proof of Theorem \ref{thmE})}\label{s:PP&IIPP}

After the characterization of groups (not) IIPP in the preceding section (whose topology fits in the concept of Theorem \ref{thmC}), we now give further non-trivial examples of groups presentable by products but not IIPP, which fit in the concept of Theorem \ref{thmB} as well. These examples include irreducible fundamental groups of low-dimensional solvable manifolds with infinite center. To this end, we characterize in terms of the IIPP property the fundamental groups of all aspherical geometric manifolds in dimensions $\leq 4$. After Theorem \ref{thmD}, the most prominent examples in the present section (that are not covered by Theorem \ref{thmD}) will be irreducible fundamental groups of nilpotent manifolds, because the quotients of these groups by their center are still presentable by products (being again nilpotent and torsion-free). We begin this section with a brief review of Thurston's geometries.

\subsection{Enumeration of the low-dimensional geometries}\label{ss:enumeration4-mfds}

Let $\mathbb{X}^n$ be a complete simply connected $n$-dimensional Riemannian manifold. We say that a closed manifold $M$ is an
{\em $\mathbb{X}^n$-manifold}, or that {\em $M$ possesses the $\mathbb{X}^n$ geometry} in the sense of Thurston, if it is diffeomorphic to a quotient of
$\mathbb{X}^n$ by a lattice $\Gamma$ in the group of isometries of $\mathbb{X}^n$ (acting effectively and transitively). The group $\Gamma$ denotes the
fundamental group of $M$. We say that $\mathbb{X}^n$ and $\mathbb{Y}^n$ are the same geometries if there exists a diffeomorphism $\psi \colon \mathbb{X}^n
\longrightarrow \mathbb{Y}^n$ and an isomorphism $\mathrm{Isom}(\mathbb{X}^n) \longrightarrow \mathrm{Isom}(\mathbb{Y}^n)$ mapping each $g \in \mathrm{Isom}(\mathbb{X}^n)$ to $\psi \circ g \circ \psi^{-1} \in \mathrm{Isom}(\mathbb{Y}^n)$.

In dimension 1, the circle is the only closed manifold being a quotient of the real line $\R$. In dimension 2, a closed surface possesses one of the geometries $S^2$, $\R^2$ or $\mathbb{H}^2$. In dimension 3, Thurston~\cite{Thurstonbook} proved that there exist eight (homotopically unique) geometries, namely the geometries $\mathbb{H}^3$, $Sol^3$,
$\widetilde{SL_2}$, $\mathbb{H}^2 \times \R$, $Nil^3$, $\R^3$, $S^2 \times \R$ and $S^3$ (see also~\cite{Scott:3-mfds}). 

The classification of the 4-dimensional geometries is due to
Filipkiewicz~\cite{Filipkiewicz}. According to that, there exist
eighteen geometries in dimension four with compact representatives. There is an additional geometry which, however, cannot be realized by any compact $4$-manifold. Here, we deal only with the aspherical geometries, because the non-aspherical ones are not interesting for
domination by products. Namely, the non-aspherical geometries are either products of a sphere with a non-compact factor ($\mathbb{H}^2 \times S^2$, $\R^2\times S^2$, $S^3 \times \R$), or compact themselves ($S^2 \times S^2$, $\mathbb{CP}^2$, $S^4$), and all of their representatives are dominated by products \cite{Hillman,KotschickLoeh1,Neofytidis}.

\subsection*{Enumeration of the aspherical geometries in dimension four}

We enumerate the aspherical 4-dimensional geometries, following Wall's papers~\cite{Wall:geom4-mfds1} and~\cite{Wall:geom4-mfds2}. Our list is adapted to the domination-by-products question, and this list will be used as an organizing principle; see Table \ref{table:4geom}.

\begin{table}
\centering
{\small
\begin{tabular}{c|c}
Type & Geometry $\mathbb{X}^4$\\
\hline
& \\
Hyperbolic & $\mathbb{H}^4$, $\mathbb{H}^2(\mathbb{C})$\\&\\
          & $\mathbb{H}^3\times\mathbb{R}$, $Sol^3\times\R$, \\
Product   & $\widetilde{SL_2}\times\mathbb{R}$, $Nil^3\times\mathbb{R}$, \\
          & $\mathbb{H}^2\times\mathbb{R}^2$, $\R^4$, \\
          & $\mathbb{H}^2\times\mathbb{H}^2$ \\&\\
Solvable    & $Nil^4$, \\
non-product & $Sol^4_{m \neq n}$, $Sol^4_0$\\
            & $Sol^4_1$\\
\empty\\
\end{tabular}}
\caption{{\small The $4$-dimensional aspherical Thurston geometries with compact representatives.}}\label{table:4geom}
\end{table}

\medskip

\noindent{{\bf Hyperbolic geometries.}} There exist two aspherical irreducible symmetric geometries, namely the real and the complex hyperbolic, denoted
by $\mathbb{H}^4$ and $\mathbb{H}^2(\C)$ respectively.

\medskip

\noindent{{\bf Product geometries.}} Seven of the aspherical geometries are products of lower dimensional geometries: $\mathbb{H}^3 \times \R$,
$Sol^3 \times \R$,
$\widetilde{SL_2} \times \R$, $Nil^3 \times \R$, $\mathbb{H}^2 \times \R^2$, $\R^4$ and $\mathbb{H}^2 \times \mathbb{H}^2$. Closed manifolds possessing a geometry of type $\mathbb{X}^3 \times \R$ satisfy the following property:

\begin{thm}[\normalfont{\cite[Sections 8.5 and 9.2]{Hillman}}]\label{t:hillmanproducts}
 Let $\mathbb{X}^3$ be a $3$-dimensional aspherical geometry. A closed $4$-manifold possessing the geometry $\mathbb{X}^3 \times \R$ is finitely
covered by a product $N \times S^1$, where $N$ is a closed oriented $3$-manifold carrying the geometry $\mathbb{X}^3$.
\end{thm}

The geometry $\mathbb{H}^2 \times \mathbb{H}^2$ can be realized both by manifolds that are virtual products of two closed hyperbolic
surfaces and by manifolds that are not even (virtual) surface bundles. These two types are known as the {\em reducible} and the {\em irreducible}
$\mathbb{H}^2 \times \mathbb{H}^2$ geometry respectively; see~\cite[Section 9.5]{Hillman} for further details and characterizations.

\medskip

\noindent{{\bf Solvable non-product geometries.}} Finally, there exist four aspherical non-product geometries of solvable type. Below, we describe their
model Lie
groups.

The nilpotent Lie group $Nil^4$ is defined as the semi-direct product $\R^3 \rtimes \R$, where $\R$ acts on $\R^3$ by
$
t \mapsto 
\exp\left(\begin{array}{ccc}
   0 & t & 0 \\
   0 & 0 & t \\
   0 & 0 & 0   \\
\end{array} \right).
$
The model spaces for the three non-product solvable -- but not nilpotent -- geometries are defined as follows:

Let $m$ and $n$ be positive integers and $a > b > c$ reals such that $a+b+c=0$ and $e^a,e^b,e^c$ are
roots of the equation $P_{m,n}(\lambda)=\lambda^3-m\lambda^2+n\lambda-1=0$. If $m \neq n$, the Lie group $Sol_{m \neq n}^4$ is defined as $\R^3 \rtimes
\R$, where $\R$ acts on $\R^3$ by
$
t \mapsto 
\left(\begin{array}{ccc}
   e^{at} & 0 & 0 \\
   0 & e^{bt} & 0 \\
   0 & 0 & e^{ct} \\
\end{array} \right).
$
We remark that the case $m=n$ gives $b = 0$ and corresponds to the product geometry $Sol^3 \times \R$. 

If we require two equal roots of the polynomial $P_{m,n}$, then we obtain the model space of the $Sol_0^4$ geometry, again defined as $\R^3
\rtimes \R$, where now the action of $\R$ on $\R^3$ is given by
$
t \mapsto 
\left(\begin{array}{ccc}
   e^{t} & 0 & 0 \\
   0 & e^{t} & 0 \\
   0 & 0 & e^{-2t} \\
\end{array} \right).
$

The last solvable model space is an extension of $\R$ by the $3$-dimensional Heisenberg group
 $Nil^3$. Namely, the Lie group $Sol_1^4$ is defined as the semi-direct product $Nil^3 \rtimes \R$, where $\R$ acts on $Nil^3$ by
$
t \mapsto 
\left(\begin{array}{ccc}
   1 & e^{-t}x & z \\
   0 & 1 & e^{t}y \\
   0 & 0 & 1 \\
\end{array} \right).
$

\medskip

Every closed $Sol_0^4$- or $Sol_{m \neq n}^4$-manifold is a mapping torus of a self-homeomorphism of $T^3$ and every closed oriented $Nil^4$- or $Sol_1^4$-manifold is a mapping torus of a self-homeomorphism of a $Nil^3$-manifold~\cite[Sections 8.6 and 8.7]{Hillman}. We note that non-orientable closed $Nil^4$- or $Sol_1^4$-manifolds are not mapping tori of $Nil^3$-manifolds~\cite[Theorem 8.9]{Hillman}. Further details about manifolds possessing a solvable non-product geometry, in particular concerning their fundamental groups, will be
provided while examining each geometry individually. 

\medskip

A crucial property for our study is that the $4$-dimensional geometries are homotopically unique by a result of Wall; cf. \cite[Theorem 10.1]{Wall:geom4-mfds2} and \cite[Prop. 1]{Kotschick:4-mfds}. In particular, {\em a closed aspherical geometric $4$-manifold $M$ is finitely covered by a closed $\mathbb{X}^4$-manifold if and only if it possesses the geometry $\mathbb{X}^4$}. 

\subsection{Proof of Theorem \ref{thmE} in dimensions${\bf\leq 3}$}\label{s:proofE}

Dimension 1 gives already the first (trivial) example of a solvable group presentable by products but not IIPP, namely the infinite cyclic group, which is the fundamental group of $S^1$. In dimension 2, the only non-trivial solvable fundamental group of an oriented manifold is the fundamental group of $T^2$ which is the product $\Z\times\Z$. The fundamental groups of higher genus surfaces are (non-elementary) hyperbolic and therefore not presentable by products as shown in~\cite{KotschickLoeh1}.
 
We now deal with infinite fundamental groups of closed $3$-manifolds. First, using Epstein's factorization theorem for $3$-manifold groups~\cite{Epstein} and the fact that not virtually cyclic free products are not presentable by products~\cite{KotschickLoeh2}, we obtain that the fundamental group of a closed $3$-manifold is presentable by products if and only if it has virtually infinite center \cite[Theorem 8]{KotschickNeofytidis}. 
These two properties are moreover equivalent to $M$ being Seifert fibered (with infinite fundamental group), by the
Seifert fiber space conjecture, which was independently proven by Gabai 
and by Casson-Jungreis. 
Recall that a closed $3$-manifold $M$ (possibly with finite fundamental group) is Seifert fibered if and
only if it is virtually a circle bundle over a closed oriented surface, by the works of Seifert, Thurston and Scott; cf.~\cite{Thurstonbook,Scott:3-mfds}. Equivalently, $M$ carries one of the geometries $\widetilde{SL_2}$, $\mathbb{H}^2 \times \R$, $Nil^3$, $\R^3$, $S^2 \times \R$ or $S^3$. Thus, we have the following consequence of \cite[Theorem 8]{KotschickNeofytidis}:

\begin{prop}[\cite{KotschickNeofytidis}]\label{p:3-mfdsPPSeifert}
 Suppose $M$ is a closed $3$-manifold with infinite fundamental group. Then $\pi_1(M)$ is presentable by products if and only if $M$ is a virtual
circle bundle over a closed oriented surface. Equivalently, $M$ possesses one of the geometries $\widetilde{SL_2}$, $\mathbb{H}^2 \times \R$, $Nil^3$ or $\R^3$.
\end{prop}

However, if the circle fiber of $M$ is not (virtually) a direct factor, i.e. if $M$ is not modeled on $\mathbb{H}^2 \times \R$ or $\R^3$, then $\pi_1(M)$ cannot be IIPP:

\begin{prop}\label{p:3-mfdsnotIIPP}
 The fundamental group of a non-trivial circle bundle $M$ over a closed oriented aspherical surface $\Sigma$ is not IIPP.
\end{prop}
\begin{proof}
 If $\Sigma$ is hyperbolic, then $\pi_1(M)$ fits into a non-split central extension
\[
 1 \longrightarrow \Z \longrightarrow \pi_1(M) \longrightarrow \pi_1(\Sigma) \longrightarrow 1.
\]
In particular, $\pi_1(M)$ fulfills the conditions of Theorems \ref{thmC} and \ref{thmD}, because non-virtually cyclic hyperbolic groups are not
presentable by products~\cite{KotschickLoeh1}. By Epstein's factorization theorem \cite{Epstein} and Stallings fibering criterion \cite{Stallings:fiber}, the fundamental group of a non-trivial circle bundle over a closed oriented surface is never reducible, and so Theorem \ref{thmD} implies that $\pi_1(M)$ is not IIPP. 

The remaining case is when $\Sigma$ has genus one, i.e. when $M$ is a non-trivial circle bundle over $T^2$ (and therefore a $Nil^3$-manifold). In
that case, $\pi_1(M)$ fits into a non-split central extension
\[
 1 \longrightarrow \Z \longrightarrow \pi_1(M) \longrightarrow \Z^2 \longrightarrow 1,
\]
where $C(\pi_1(M)) = \Z$. Suppose that $\pi_1(M)$ is IIPP. Then we may assume that there exist non-trivial infinite-index commuting subgroups $\Gamma_1,
\Gamma_2
\subset \pi_1(M)$ and a short exact sequence
\[
 1 \longrightarrow \Gamma_1\cap\Gamma_2 \longrightarrow \Gamma_1 \times \Gamma_2 \stackrel{\varphi}\longrightarrow \pi_1(M) \longrightarrow 1,
\]
where $\varphi$ is the multiplication map and $\Gamma_1 \cap \Gamma_2 \subset C(\pi_1(M))$. (Note that both $\Gamma_i$ are torsion-free, because $\pi_1(M)$
is torsion-free.) We observe that $\Gamma_1 \cap \Gamma_2$ cannot be trivial, otherwise $\pi_1(M)$ would be Abelian (equivalently, $M$ would be $T^3$). This means that $\Gamma_1 \cap \Gamma_2$ must be isomorphic to
$C(\pi_1(M)) = \Z$. Moreover, since $[\pi_1(M) \colon \Gamma_i] = \infty$ and $\pi_1(M)$ has cohomological
dimension three, we conclude that each of the $\Gamma_i$ is of cohomological dimension at most two~\cite{Sterbel}. Now, $\Gamma_1 \cap \Gamma_2$ is central in both $\Gamma_i$ which means that the quotients $\Gamma_i/(\Gamma_1 \cap \Gamma_2)$ are finitely
generated and virtually free groups
$F_{k_i}$, by a result of Bieri~\cite[Cor. 8.7]{Bieri:book}. Passing to finite coverings, we may assume that these quotient groups are free and therefore
the central extensions
\[
 1 \longrightarrow \Gamma_1 \cap \Gamma_2 \longrightarrow \Gamma_i \longrightarrow F_{k_i} \longrightarrow 1 
\]
split. We have finally reached the absurd conclusion that $\pi_1(M)$ is virtually isomorphic to a direct product $\Z \times F_{k_1} \times F_{k_2}$.
\end{proof}

\begin{rem}\label{r:hirschlength}
 The case of nilpotent groups in the above proposition could be treated in a different way, using the Hirsch length, 
 however, the dimensions here were suitable to appeal to Bieri's~\cite{Bieri:book} result on central extensions. In fact, the cohomological dimension of a finitely generated torsion-free nilpotent group coincides to its Hirsch length~\cite{Gruenberg}.  
We will return and use the Hirsch length of nilpotent groups in Section \ref{ss:4-mfdsIIPP}.
\end{rem}

Because the fundamental group of $S^2 \times S^1$ is infinite cyclic and therefore not IIPP, we have now determined all fundamental groups of closed $3$-manifolds that are (not) IIPP: 

\begin{cor}\label{c:IIPPequivalence3-mfds}
 Suppose that the fundamental group of a closed oriented $3$-manifold $M$ is infinite. Then $\pi_1(M)$ is IIPP if and only if it is reducible. Equivalently, $\pi_1(M)$ is a virtual product $\pi_1(\Sigma) \times \Z$, where $\Sigma$ is a closed oriented aspherical surface.
\end{cor}

In particular, we have proved Theorem \ref{thmE} in dimension three.

\subsection{Proof of Theorem \ref{thmE} in dimension four}\label{s:4-mfdsnilsol}

We will prove Theorem \ref{thmE} in dimension four in two steps. First, we will determine which closed aspherical geometric $4$-manifolds have fundamental groups (not) presentable by products:

\begin{thm}\label{t:4-mfdgroupsPP}
  The fundamental group of a closed aspherical geometric $4$-manifold $M$ is presentable by products if and only if
$M$ possesses one of the geometries $\mathbb{X}^3\times \R$, $Nil^4$, $Sol^4_1$ or the reducible $\mathbb{H}^2\times\mathbb{H}^2$ geometry.
\end{thm}

Then, Theorem \ref{thmE} will follow as a consequence of the next characterization:

\begin{thm}\label{t:4-mfdgroupsIIPP}
  The fundamental group of a closed aspherical geometric $4$-manifold $M$ is reducible if and only if it is IIPP.
Equivalently, $M$ carries one of the product geometries $\mathbb{X}^3 \times \R$ or the reducible $\mathbb{H}^2\times\mathbb{H}^2$ geometry.
\end{thm}

\subsection*{Groups presentable by products: Proof of Theorem \ref{t:4-mfdgroupsPP}}\label{ss:4-mfdsPP}

We proceed by examining case by case all aspherical geometries, following the enumeration given at the beginning of this section (cf. Table \ref{table:4geom}).

\subsubsection{Hyperbolic geometries.}
 
As we have already seen (and used in the proof of Proposition \ref{p:3-mfdsnotIIPP}), non-virtually
cyclic hyperbolic groups are not presentable by products; we refer to~\cite[Prop. 3.6]{KotschickLoeh1} for the details. Since the fundamental groups of closed $4$-manifolds with hyperbolic geometries $\mathbb{H}^4$ or $\mathbb{H}^2(\C)$ are obviously not (virtually) infinite cyclic, we deduce that they are not presentable by products. 

\subsubsection{Product geometries.}

An equivalent formulation of Theorem \ref{t:hillmanproducts} is the following:

\begin{cor}
 The fundamental group of a closed aspherical $4$-manifold $M$ carrying a product geometry $\mathbb{X}^3 \times \R$ is a virtual product $\pi_1(N) \times
\Z$, where $N$ is a closed aspherical $\mathbb{X}^3$-manifold. In particular, $\pi_1(M)$ is presentable by products and $C(\pi_1(M))$ is
virtually infinite.
\end{cor}

As mentioned previously, the geometry $\mathbb{H}^2 \times \mathbb{H}^2$ is an exceptional type among the product geometries, because not every closed $\mathbb{H}^2 \times \mathbb{H}^2$-manifold has a finite cover which is a product of two closed hyperbolic
surfaces. This property distinguishes closed $\mathbb{H}^2 \times \mathbb{H}^2$-manifolds into two classes, the reducible and the irreducible ones. Since $\mathbb{H}^2 \times \mathbb{H}^2$-manifolds admit metrics of non-positive sectional curvature, Theorem 4.1 of \cite{KotschickLoeh1} implies that
irreducible lattices in the group of isometries of $\mathbb{H}^2 \times \mathbb{H}^2$ are not presentable by products:

\begin{prop}\label{p:h2xh2PP}
 The fundamental group of a closed $\mathbb{H}^2 \times \mathbb{H}^2$-manifold $M$ is presentable by products if and only if it is a virtual product of
two closed hyperbolic surface groups. Equivalently, $M$ carries the reducible $\mathbb{H}^2 \times \mathbb{H}^2$ geometry.
\end{prop}

\subsubsection{Solvable non-product geometries.}

We finally deal with solvable non-product geometries, i.e. the geometries $Nil^4$, $Sol^4_{m\neq n}$, $Sol^4_0$ and $Sol_1^4$. As we shall see, only
the fundamental groups of closed manifolds modeled on the geometries $Nil^4$ or $Sol_1^4$ are presentable by products.

\medskip

\subparagraph{{\em The geometry $Nil^4$.}}

\begin{prop}\label{p:nil4-mfds}
 A closed $Nil^4$-manifold $M$ is a virtual circle bundle over a closed oriented $Nil^3$-manifold and the center of $\pi_1(M)$ remains infinite cyclic in finite covers.
 \end{prop}
\begin{proof}
 Let $M$ be a closed $Nil^4$-manifold. After possibly passing to a double cover, we may assume that $M$ is oriented and so $\pi_1(M)$ fits into a short
exact sequence 
\[
 1 \longrightarrow \pi_1(N) \longrightarrow \pi_1(M) \longrightarrow \Z \longrightarrow 1,
\]
where $N$ is a closed oriented $Nil^3$-manifold and a generator $t \in \Z$ acts by conjugation on $\pi_1(N)$; cf. \cite[Sections 8.6 and 8.7]{Hillman}. Passing to another finite cover, if necessary, we may assume that $N$ is a non-trivial circle bundle over $T^2$ with fundamental
group
\[
 \pi_1(N) = \langle x,y,z \ \vert \ [x,y] = z, \ xz = zx, \ yz = zy \rangle,
\]
where $C(\pi_1(N)) = \langle z \rangle$; cf.~\cite{Scott:3-mfds}.

Since $M$ is a $Nil^4$-manifold, the automorphism of $\pi_1(N)/\langle z \rangle \ \cong \ \Z^2$, induced by the action of $t \in \Z$ on $\pi_1(N)$, is
given (after
possibly passing to another finite cover) by a matrix (conjugate to)
$A = \left( \begin{array}{cc}
  1 & k \\
  0 & 1 \\
\end{array} \right) \in \mathrm{GL}_2(\Z)$, for some $k \neq 0$; cf.~\cite[Theorem 8.7]{Hillman}. The relation $x^my^n =
z^{mn}y^nx^m$ in $\pi_1(N)$ gives the following presentation of $\pi_1(M)$ (see also~\cite{SakamotoFukuhara} and~\cite[pg. 522]{Ue1} for further details):
 \[
 \pi_1(M) = \langle x,y,z,t \ \vert \ txt^{-1}=x, \ tyt^{-1}=x^kyz^l, \ tzt^{-1} = z^{\det A} = z,
              [x,y]=z, \ xz=zx, \ yz=zy \rangle,
 \]
where $C(\pi_1(M)) = \langle z \rangle$. Thus we have a short exact sequence 
\begin{equation}\label{eq.nil4-mfds}
 1 \longrightarrow \langle z \rangle \longrightarrow \pi_1(M) \longrightarrow Q \longrightarrow 1,
\end{equation}
where $Q = \pi_1(M)/\langle z \rangle = \langle x,y,t \ \vert \ [t,y]=x^k, \ xt =tx, \ xy=yx \rangle$. In particular, the classifying space $BQ$ is a
non-trivial circle bundle over $T^2$ and thus a $Nil^3$-manifold. Now, the induced sequence of the classifying spaces corresponding to
(\ref{eq.nil4-mfds}) implies that $M$ is homotopically a circle bundle over $BQ$. Finally, the center of $\pi_1(M)$ remains infinite cyclic in finite covers, generated by multiples of $z$, because $k \neq 0$.
\end{proof}

\begin{rem}
Since every non-trivial nilpotent group has non-trivial center and because the property of being ``nilpotent'' is closed under subgroups and quotient groups, the proof of the
above proposition could be obtained using the fact that every nilpotent group of cohomological dimension three is either Abelian or isomorphic to $Q$ (as
in the above proof); see Proposition \ref{p:nilproperties}, Lemma \ref{l:cd=3nilpotent} and Remark \ref{r:cd=3}.  
\end{rem}

\begin{rem}\label{r:nil4mappingtorus}
Note that $\pi_1(M)$ is (virtually) an extension of $\Z^2 = \langle y,t \rangle$ by $\Z^2 = \langle z,x \rangle$, and so $M$ is (virtually) a
$T^2$-bundle over $T^2$, whose $T^2$-fiber contains the $S^1$-fiber of the circle bundle $S^1 \longrightarrow M \longrightarrow BQ$ of
the above proposition. It is a result of Ue~\cite[Theorem B]{Ue1} that every closed $Nil^4$-manifold
is a virtual $T^2$-bundle over $T^2$. We refer to a work of Fukuhara and Sakamoto~\cite{SakamotoFukuhara} for a
classification of $T^2$-bundles over $T^2$.

 We furthermore observe that $\pi_1(M)$ is (virtually) an extension of $\Z = \langle y \rangle$ by $\Z^3 = \langle z,x,t \rangle$, where the
automorphism of $\Z^3$ is given by 
$\left(\begin{array}{ccc}
   1 & -1 & -l \\
   0 & 1 & -k \\
   0 & 0 & 1 \\
\end{array} \right)$
and it has infinite order. In particular, $M$ is a (virtual) mapping torus of $T^3$; see also~\cite[Section 8.6]{Hillman}.
\end{rem}


\subparagraph{{\em The geometries $Sol^4_{m\neq n}$ and $Sol^4_0$.}} 

\begin{prop}\label{p:solvablenotPP}
 The fundamental group of a closed $4$-manifold possessing one of the geometries $Sol^4_{m\neq n}$ or $Sol^4_0$ is not presentable by products.
\end{prop}

In order to prove Proposition \ref{p:solvablenotPP}, we will use the concept of ``acentral" subgroups which was introduced in~\cite{KotschickLoeh2}: A subgroup $A$ of a group $\Gamma$ is called {\em acentral} if for every non-trivial element $g \in A$ the centralizer $C_{\Gamma}(g)$ is contained in $A$. An {\em acentral extension} is an extension of groups $1 \longrightarrow N \longrightarrow \Gamma \longrightarrow Q \longrightarrow 1$ such that the normal subgroup $N$ is acentral.

\begin{prop}[\normalfont{\cite[Prop. 3.2]{KotschickLoeh2}}]\label{p:kotschickloehacentral}
If a group contains an infinite acentral subgroup of infinite index, then it is not presentable by products.
\end{prop}

Proposition \ref{p:kotschickloehacentral} implies that if an extension $1 \longrightarrow N \longrightarrow \Gamma \longrightarrow Q
\longrightarrow 1$ is acentral, with $N$ and $Q$ infinite, then $\Gamma$ is not presentable by products \cite[Cor. 3.3]{KotschickLoeh2}. Also, if 
$\Gamma$ is a semi-direct product $N \rtimes_{\theta} Q$,
where $N$ is non-trivial Abelian and $Q$ is infinite acting freely outside $0\in N$, 
then the extension $0 \longrightarrow N \longrightarrow \Gamma \longrightarrow Q \longrightarrow
1$ is acentral and $N$ is infinite. In particular, $\Gamma$ is not presentable by products \cite[Cor. 3.5]{KotschickLoeh2}. This gives an alternative proof that $Sol^3$-manifold groups are not presentable by products \cite[Section 3]{KotschickLoeh2}.

\begin{proof}[Proof of Proposition \ref{p:solvablenotPP}]
 We will show that the fundamental groups of closed $Sol^4_{m\neq n}$- or $Sol^4_0$-manifolds contain acentral subgroups of infinite index. 

Every manifold $M$ with geometry modeled on $Sol^4_{m\neq n}$ or $Sol^4_0$ is a mapping torus of a
self-homeomorphism of
$T^3$ \cite[Section 8.6]{Hillman} (see also~\cite{Wall:geom4-mfds1,Wall:geom4-mfds2}) and its fundamental group is a semi-direct product $\Z^3 \rtimes_{\theta} \Z$, where the automorphism $\theta$ of $\Z^3$ is induced by the action by conjugation of a generator $t \in \Z$.

Now, if $M$ is a $Sol^4_{m\neq n}$-manifold, then $\theta$ has three real distinct eigenvalues and none of them is equal to $\pm 1$, because $M$ is neither
nilpotent nor carries the $Sol^3 \times \R$ geometry (which is the case $m = n$); cf.~\cite{Wall:geom4-mfds1} and~\cite[pg. 164/165]{Hillman} (or Section
\ref{ss:enumeration4-mfds}).

If $M$ is a $Sol^4_0$-manifold, then $\theta$ has two complex eigenvalues that are not roots of unity and a real eigenvalue not equal to $\pm 1$;
cf.~\cite{Wall:geom4-mfds1} and~\cite[pg. 164/165]{Hillman} (or Section \ref{ss:enumeration4-mfds}).

In both cases we derive that the centralizer $C_{\pi_1(M)}(g)$ of each $g \in \Z^3 \setminus \{0\}$ is contained in $\Z^3$ (it is actually equal to
$\Z^3$). This means that the infinite-index normal subgroup $\Z^3$ is acentral and so $\pi_1(M)$ is
not presentable by products by Proposition
\ref{p:kotschickloehacentral} or \cite[Cor 3.3 or Cor 3.5]{KotschickLoeh2}. 
\end{proof}

\subparagraph{{\em The geometry $Sol_1^4$.}}

\begin{prop}\label{p:sol14-mfds}
 A closed $Sol_1^4$-manifold $M$ is a virtual circle bundle over a mapping torus of $T^2$ with hyperbolic monodromy. In particular, $\pi_1(M)$ is presentable by products.
\end{prop}
\begin{proof}
 Let $M$ be a closed $Sol_1^4$-manifold. After passing to a double cover, we may assume that $M$ is oriented, and so
its fundamental group fits into a short exact sequence
\[
 1 \longrightarrow \pi_1(N) \longrightarrow \pi_1(M) \longrightarrow \Z \longrightarrow 1,
\]
where $N$ is a closed oriented $Nil^3$-manifold and a generator $t \in \Z$ acts by conjugation on $\pi_1(N)$; cf. \cite[Sections 8.6 and 8.7]{Hillman}. If necessary, we pass to another finite cover of $M$ and so we can assume that the fiber $N$ is a non-trivial circle bundle over $T^2$ and that its
fundamental group has a presentation 
\[
\pi_1(N) = \langle x,y,z \ \vert \ [x,y]=z, \ xz=zx, \ yz=zy \rangle
\] 
with center $C(\pi_1(N)) = \langle z \rangle$; cf.~\cite{Scott:3-mfds}.  

Let 
$A = \left( \begin{array}{cc}
  a & b \\
  c & d \\
\end{array} \right) \in \mathrm{GL}_2(\Z)$ be the automorphism of $\pi_1(N)/\langle z \rangle \cong \Z^2$ induced by the action of
$t \in \Z$ on $\pi_1(N)$. The eigenvalues $\lambda_1,\lambda_2$ of $A$ satisfy $\det(A) = \lambda_1\lambda_2 = \pm 1$. Actually, $\det A = 1$,
because $M$ is oriented. Moreover, $\lambda_i \neq \pm 1$, because $\pi_1(M)$ is not nilpotent; see also~\cite[Theorem 8.7]{Hillman}. We conclude that $A$ is a hyperbolic
automorphism.

Now the relation $x^my^n = z^{mn}y^nx^m$ in $\pi_1(N)$ implies that a presentation of the fundamental group of $M$ is given by
\begin{eqnarray*}
  \pi_1(M) = &\langle x,y,z,t \ \vert & txt^{-1}=x^ay^cz^k, \ tyt^{-1}=x^by^dz^l, \ tzt^{-1} = z^{\det A} = z,\\
             &\ & [x,y]=z, \ xz=zx, \ yz=zy \rangle, \ k,l \in \Z,
\end{eqnarray*}
where the infinite cyclic group generated by $z$ is central in $\pi_1(M)$. Thus we obtain a short exact sequence
\begin{equation}\label{eq.sol14-mfds}
 1 \longrightarrow \langle z \rangle \longrightarrow \pi_1(M) \longrightarrow Q \longrightarrow 1,
\end{equation}
where $Q = \pi_1(M)/\langle z \rangle = \langle x,y,t \ \vert \ txt^{-1}=x^ay^c, \ tyt^{-1}=x^by^d, \ xy=yx \rangle$.  Clearly, the group $Q$ fits into an
extension
\[
 1 \longrightarrow \Z^2 \longrightarrow Q \longrightarrow \Z \longrightarrow 1,
\]
where $t \in \Z$ acts on $\Z^2$ by the hyperbolic automorphism $A  = \left( \begin{array}{cc}
  a & b \\
  c & d \\
\end{array} \right)$. We have now shown that $Q$ is the fundamental group of a mapping torus of $T^2$ with
hyperbolic
monodromy, i.e. $BQ$ is a closed oriented $Sol^3$-manifold. Therefore, $M$ is homotopically a circle bundle over $BQ$, by the induced sequence in homotopy corresponding to
the short exact sequence (\ref{eq.sol14-mfds}).
\end{proof}

The proof of Theorem \ref{t:4-mfdgroupsPP} is now complete.

\subsection*{Groups not IIPP: Proof of Theorem \ref{t:4-mfdgroupsIIPP}}\label{ss:4-mfdsIIPP}

Since reducible groups are IIPP, in order to complete the proof of Theorem \ref{t:4-mfdgroupsIIPP} we need to show that the fundamental groups of manifolds modeled on
$Nil^4$ or $Sol_1^4$ are not IIPP. 

\medskip

\subparagraph{{\em Closed $Nil^4$-manifolds.}}

Since torsion-free nilpotent groups have infinite center, an immediate consequence of Proposition \ref{p:nil4-mfds} (see also Proposition \ref{p:nilproperties}) is the following: 

\begin{lem}\label{l:nil4-mfdsnotreducible}
 The fundamental group of a closed $Nil^4$-manifold is not a virtual product. 
\end{lem}

In the preceding section (Proposition \ref{p:3-mfdsnotIIPP}), we showed that closed $Nil^3$-manifolds have fundamental groups not IIPP using Bieri's results~\cite{Bieri:book}, since the cohomological dimensions of those groups (and of their subgroups) were suitable for that purpose. 
Passing now one dimension higher, we cannot appeal anymore to those results. However, 
instead of the cohomological dimension, we may use the Hirsch length (which is in fact equal to the cohomological dimension for finitely generated torsion-free
nilpotent groups by a theorem of Gruenberg, cf. \cite[$\S$8.8]{Gruenberg}). 
The Hirsch length generalizes the notion of the rank of free Abelian groups:

\begin{defn}[\cite{Gruenberg}]
 Let $\Gamma$ be a (virtually) polycyclic group with a series 
\[
 \Gamma = \Gamma_0 \supset \Gamma_1 \supset \cdots \supset \Gamma_n = 1,
\]
so that the quotients $\Gamma_i/\Gamma_{i+1}$ are cyclic. The sum of the ranks of these quotients is independent of the choice of the series of groups
and is called the {\em Hirsch length} of $\Gamma$. We denote the Hirsch length of $\Gamma$ by $h(\Gamma)$.
\end{defn}

\begin{ex}\label{ex:Hirschnil}
A finitely generated torsion-free nilpotent group $\Gamma$ is polycyclic, admitting a central series
$
 \Gamma = \Gamma_0 \supset \Gamma_1 \supset \cdots \supset \Gamma_n = 1,
$
 so that the quotients $\Gamma_i/\Gamma_{i+1}$ are infinite
cyclic for all $i=0,1,...,n-1$.
Therefore $\Gamma$ has a well-defined Hirsch length (equal to $n$).
\end{ex}

The following proposition gives some basic properties of nilpotent groups and their Hirsch length. For a proof see \cite[pg. 75--77]{Neothesis}.

\begin{prop}\label{p:nilproperties}
Let $\Gamma$ be a finitely generated nilpotent group.
\begin{itemize}
\item[(1)] If $\Gamma$ is torsion-free, then $C(\Gamma)$ has positive rank and $\Gamma/C(\Gamma)$ is torsion-free. 
\item[(2)] If $K$ is a normal subgroup of $\Gamma$, then $K$ and $\Gamma/K$ are finitely generated nilpotent and \newline
(a) $h(\Gamma)=h(K)+h(\Gamma/K)$;
(b) $h(K)=h(\Gamma)$ if and only if $[\Gamma : K] < \infty$.
\end{itemize}
\end{prop}

Using Proposition \ref{p:nilproperties}, we determine all torsion-free nilpotent groups of Hirsch length three. 

\begin{lem}\label{l:cd=3nilpotent}
 A torsion-free nilpotent group $\Gamma$ of Hirsch length three is isomorphic to 
 $
 G_n := \langle x, y, z \ \vert \ zy=yz, zx=xz, [x,y]=z^n \rangle,
$
 for some $n \geq 0$. In
particular, $\Gamma$ is the fundamental group of a circle bundle over $T^2$.
\end{lem}
\begin{proof}
 First, we observe that $\Gamma$ is finitely generated, because it is nilpotent of finite Hirsch length. Moreover, since $\Gamma$ is torsion-free, we have that its center $C(\Gamma)$ is free Abelian of positive rank, the quotient
group $Q:= \Gamma/C(\Gamma)$ is again nilpotent and torsion-free and the short exact sequence 
\begin{equation}\label{eq.hirsch}
 1 \longrightarrow C(\Gamma) \longrightarrow \Gamma \longrightarrow Q \longrightarrow 1,
\end{equation}
yields that $0 \leq h(Q) \leq 2$, because $h(C(\Gamma)) \geq 1$; cf Proposition \ref{p:nilproperties}. 

If $h(Q)=0$ or $1$, then it is easy to see that $\Gamma$ is free Abelian of rank three. Indeed, this is obvious if $h(Q)=0$. If
$h(Q)=1$, then $Q$ is infinite cyclic, and therefore the central extension (\ref{eq.hirsch}) splits. Suppose, finally, that $h(Q)=2$. Since $Q$ is torsion-free nilpotent, it has non-trivial center $C(Q)$. Therefore, it fits into a short exact sequence 
 \[
 1 \longrightarrow C(Q) \longrightarrow Q \longrightarrow Q/C(Q) \longrightarrow 1,
\]
where the quotient $Q/C(Q)$ is again a torsion-free nilpotent group. By the additivity of the Hirsch length for the above exact sequence, we deduce that
$h(Q/C(Q) \leq 1$. This finally implies that $Q$ is free Abelian of rank two. Therefore, the central extension (\ref{eq.hirsch}) takes the form
\[
 1 \longrightarrow \Z \longrightarrow \Gamma \longrightarrow \Z^2 \longrightarrow 1.
\]
Choosing presentations $\Z = \langle z \rangle$ and $\Z^2 = \langle x,y \ \vert \ [x,y]=1 \rangle$, we deduce that $\Gamma$ is isomorphic to $G_n$ for some
$n \geq 0$.
\end{proof}

\begin{rem}\label{r:cd=3}
In the light of Gruenberg's theorem~\cite[$\S$8.8]{Gruenberg}, Lemma \ref{l:cd=3nilpotent} determines all finitely generated nilpotent groups of cohomological dimension three. 
Moreover, it yields another
proof of the fact that closed $Nil^4$-manifolds are virtual circle bundles over closed
oriented $Nil^3$-manifolds; compare Proposition \ref{p:nil4-mfds}.
\end{rem}

We now finish the proof that $Nil^4$-manifold groups are not IIPP:

\begin{prop}\label{p:nil4-mfdsnotIIPP}
 The fundamental group of a closed $Nil^4$-manifold $M$ is not IIPP.
\end{prop}
\begin{proof}
We know that $\pi_1(M)$ is presentable by products and that $M$ is virtually a non-trivial circle bundle over a closed oriented $Nil^3$-manifold (cf. Proposition \ref{p:nil4-mfds}). We need to prove that $\pi_1(M)$ cannot be
presented by a product of subgroups $\Gamma_1$ and $\Gamma_2$ so that both $\Gamma_i$ have infinite index in $\pi_1(M)$. We proceed again by contradiction.
After passing to suitable finite index subgroups, suppose that there exist two infinite-index commuting subgroups $\Gamma_i \subset \pi_1(M)$ and a short
exact sequence
\begin{equation}\label{eq.circlebundlenil}
 1 \longrightarrow \Gamma_1 \cap \Gamma_2 \longrightarrow \Gamma_1 \times \Gamma_2 \stackrel{\varphi}\longrightarrow \pi_1(M) \longrightarrow 1,
\end{equation}
where $\varphi$ is the multiplication map and $\Gamma_1 \cap \Gamma_2 \subset C(\pi_1(M))$; cf. Section \ref{s:preliminariesPP}. Also we have that
$C(\pi_1(M)) = \Z$ (by Proposition \ref{p:nil4-mfds}) and that the $\Gamma_i$ are finitely generated, torsion-free and nilpotent.

Since $\pi_1(M)$ is not reducible (Lemma \ref{l:nil4-mfdsnotreducible}), we conclude that the intersection $\Gamma_1 \cap \Gamma_2$ is not trivial and so it must be infinite cyclic, as a subgroup of $C(\pi_1(M)) = \Z$. Also, Proposition \ref{p:nilproperties} and Lemma \ref{l:cd=3nilpotent} imply that the Hirsch length of $\pi_1(M)$ is $h(\pi_1(M))=4$. Applying again the additivity of the Hirsch length to the short exact sequence (\ref{eq.circlebundlenil}), we have that $h(\Gamma_1 \times \Gamma_2) = 5$. Since both $\Gamma_i$ have infinite index in $\pi_1(M)$, we deduce that $h(\Gamma_i) \leq 3$ by Proposition \ref{p:nilproperties}. Therefore, one of the $\Gamma_i$ must be of Hirsch length three and the other of Hirsch length two (again by Proposition \ref{p:nilproperties}). Let us assume that $h(\Gamma_1)=3$ and $h(\Gamma_2)=2$. Since $\Gamma_1$ is a torsion-free nilpotent group of Hirsch length three, Lemma \ref{l:cd=3nilpotent} implies that $\Gamma_1$ isomorphic to $G_n$ for some $n \geq 0$. Moreover, $\Gamma_2$ is isomorphic to $\Z^2$, because it is torsion-free nilpotent of Hirsch length two (see the proof of Lemma \ref{l:cd=3nilpotent}). We have now reached the conclusion that the rank of the center of $\Gamma_1 \times \Gamma_2$ is at least three. This is however not possible, according to the next lemma, because $\Gamma_1\times\Gamma_2$ is an extension of $\pi_1(M)$ by $\Z$ and $C(\pi_1(M)) = \Z$.
\end{proof}

\begin{lem}\label{l:extensionscenter}
 If a group $\Gamma$ with finitely generated center $C(\Gamma)$ fits into a central extension
$
1 \longrightarrow \Z^k \longrightarrow \Gamma \stackrel{\pi}\longrightarrow Q \longrightarrow 1,
$ 
where $Q$ is torsion-free, then $\mathrm{rank} C(\Gamma) \leq \mathrm{rank} C(Q) + k$.
\end{lem}
\begin{proof}
 It follows by the fact that if $x \in C(\Gamma)$, then $\pi(x) \in C(Q)$. 
\end{proof}

\subparagraph{{\em Closed $Sol_1^4$-manifolds.}}

We finally deal with the $Sol_1^4$ geometry.

\begin{prop}\label{p:sol14-mfdsIIPP}
 The fundamental group of a closed $Sol_1^4$-manifold is not IIPP.
\end{prop}
\begin{proof}
 By Proposition \ref{p:sol14-mfds}, a closed $Sol_1^4$-manifold $M$ is a virtual circle bundle over a closed oriented $Sol^3$-manifold $N$. In particular,
its fundamental group $\pi_1(M)$ satisfies all the assumptions of Theorem \ref{thmD} (recall that $\pi_1(N)$ is not presentable by
products; see Proposition \ref{p:3-mfdsPPSeifert}). Thus, $\pi_1(M)$ is IIPP if and only if it is a
virtual product $\pi_1(N) \times \Z$. The latter is impossible by Wall's uniqueness theorem \cite[Theorem 10.1]{Wall:geom4-mfds2}.
\end{proof}

This finishes the proof of Theorem \ref{t:4-mfdgroupsIIPP} and therefore the proof of Theorem \ref{thmE}.

\section{Domination by products in dimension four (proof of Theorem \ref{thmF})}

In this section we determine which geometric $4$-manifolds are (not) dominated by products. To this end, we combine the algebraic results of the previous section  
with the topological statements of Theorem \ref{t:KotschickLoehmain} and Theorems \ref{thmB} and \ref{thmC} to prove Theorem \ref{thmF}, which completes Hillman's characterization of product geometries given in Theorem \ref{t:hillmanproducts}.

\begin{proof}[Proof of theorem \ref{thmF}]
Theorem \ref{t:hillmanproducts} says that closed manifolds possessing one of the product geometries $\mathbb{X}^3 \times \R$ are finitely covered by products of type $N\times S^1$. Moreover, closed manifolds carrying the reducible $\mathbb{H}^2 \times \mathbb{H}^2$ geometry are virtually products of closed hyperbolic surfaces. In particular, all those manifolds are dominated by products. (Note that domination by products alone follows by Theorem \ref{thmA} as well.) 

By Wall's uniqueness theorem of the $4$-dimensional geometries, it suffices to show that closed $4$-manifolds possessing either a hyperbolic geometry, the irreducible $\mathbb{H}^2 \times \mathbb{H}^2$ geometry, or a non-product solvable geometry cannot be dominated by products.

For the hyperbolic geometries $\mathbb{H}^4$ and $\mathbb{H}^2(\C)$, the irreducible $\mathbb{H}^2 \times \mathbb{H}^2$ geometry, and the solvable geometries $Sol_{m \neq n}^4$ and $Sol_0^4$ our claim can be deduced by Theorem \ref{t:KotschickLoehmain}, because the fundamental groups of closed $4$-manifolds carrying one of those geometries are not presentable by products; see Section \ref{ss:4-mfdsPP} for the details.

If now $M$ is a closed $Nil^4$-manifold, then it is a virtual circle bundle over a closed oriented $Nil^3$-manifold and the center of its fundamental group remains infinite cyclic in finite covers; cf. Proposition \ref{p:nil4-mfds}. Moreover, $\pi_1(M)$ is not IIPP, by Proposition \ref{p:nil4-mfdsnotIIPP}, and so Theorem \ref{thmB} implies that $M$ cannot be dominated by products.

Finally, if $M$ is a closed $Sol_1^4$-manifold, then it is a virtual circle bundle over a closed oriented $Sol^3$-manifold and $\pi_1(M)$ is not IIPP; cf. Propositions \ref{p:sol14-mfds} and \ref{p:sol14-mfdsIIPP} respectively. Therefore, $M$ is not dominated by products by Theorem \ref{thmC}, because closed $Sol^3$-manifolds have fundamental groups not presentable by products. (Equivalently, $M$ is not dominated by products because of the equivalence between $(1)$ and $(2)$ in Theorem \ref{thmC} and by Wall's uniqueness theorem.)

The proof is now complete.
\end{proof}

Combining Theorem \ref{thmE} with the characterizations of groups (infinite-index) presentable by products in Section \ref{s:PP&IIPP}, we obtain the following purely algebraic characterization:

\begin{cor}\label{c:answerq24D}
A closed oriented aspherical geometric $4$-manifold $M$ is dominated by a product if and only if 
\begin{itemize}
 \item[\normalfont{(1)}] either $\pi_1(M)$ is a virtual product $\pi_1(N) \times \Z$, for some closed aspherical geometric $3$-manifold $N$, or
 \item[\normalfont{(2)}] $\pi_1(M)$ is a virtual product of two closed hyperbolic surface groups. 
\end{itemize} 
\end{cor}

\section{The $5$-dimensional Heisenberg manifold}\label{s:Heisenberg}

In this section we give the first example of an aspherical manifold whose fundamental group is IIPP and irreducible, and show that this manifold is not dominated by products. 

\begin{lem}\label{l:H5}
The $5$-dimensional Heisenberg group
\[
H_5=\langle x,y,u,v,z \ \vert \ [x,y]=[u,v]=z, \ \text{\normalfont{all other commute}} \rangle
\]
is IIPP and irreducible.
\end{lem}
\begin{proof}
We observe that the groups
$
\Gamma_1=\langle x,y,z \ \vert \ [x,y]=z, \ xz=zx, \ yz=zy\rangle$ and $\Gamma_2=\langle u,v,z \ \vert \ [u,v]=z, \ uz=zu, \ vz=zv\rangle$ are infinite index subgroups of $H_5$ that commute element-wise, and $\Gamma_1\Gamma_2=H_5$. In particular, $H_5$ is presented by the product $\Gamma_1\times\Gamma_2$ through the multiplication homomorphism $\Gamma_1\times\Gamma_2\longrightarrow H_5$. Thus $H_5$ is IIPP.
 
Now we show that $H_5$ is irreducible. First of all, $H_5$ is not a product itself, otherwise its center would have been of rank greater than $1$, because subgroups of nilpotent groups are nilpotent themselves and torsion-free nilpotent groups have infinite center. Second, every finite index subgroup $\overline{H_5} \subset H_5$ fits into a short exact sequence
\[
1\longrightarrow\langle z^k \rangle\longrightarrow\overline{H_5}\longrightarrow Q\longrightarrow1,
\]
where $Q$ is a finite index subgroup of $\Z^4=H_5/\langle z\rangle$. In particular, $Q$ is isomorphic to $\Z^4$, generated by some powers of $x,y,u$ and $v$. If $\overline{H_5}$ were a direct product, then the relations $[x^m,y^n]=z^{mn}$ and $[u^m,v^n]=z^{mn}$ in $\overline{H_5}$ would imply that $z$ is a torsion element. However, $H_5$ is torsion-free and the lemma follows. 
\end{proof}

The classifying space of $H_5$ is a non-trivial circle bundle over $T^4$, called the $5$-dimensional Heisenberg manifold. 

\begin{thm}\label{t:H5}
The Heisenberg manifold $BH_5$ is not dominated by products.
\end{thm}
\begin{proof}
Suppose, for contrast, that there exists a map of non-zero degree $f\colon X_1\times X_2\longrightarrow BH_5$ which is $\pi_1$-surjective, after possibly replacing $BH_5$ by a finite cover;  this cover will remain a non-trivial circle bundle over $T^4$ with center $C(H_5)\cong\Z$, by Lemma \ref{l:H5}. We have a short exact sequence
\begin{eqnarray}\label{eq.BH5}
 1 \longrightarrow \Gamma_1\cap\Gamma_2 \longrightarrow \Gamma_1 \times \Gamma_2 \stackrel{\varphi}\longrightarrow H_5 \longrightarrow 1,
\end{eqnarray}
where $\Gamma_i := \mathrm{im}(\pi_1(f\vert_{X_i})) \subset H_5$, \ $\Gamma_1\cap\Gamma_2 \subset C(H_5) = \Z$ and $\varphi$ is the multiplication homomorphism; see Sections \ref{s:notation} and \ref{s:preliminariesPP} for the details. In particular, $\Gamma_1\cap\Gamma_2$ is isomorphic to $\Z$, because $H_5$ is irreducible.
Since the $\Gamma_i$ are finitely generated torsion-free nilpotent groups, the above exact sequence implies that it suffices (up to order) to examine what happens when the Hirsch length of $\Gamma_1$ takes the values one, two and three.

\medskip
\noindent{\bf Case I.}  Suppose first that $h(\Gamma_1)=1$. Then $X_1=S^1$ (see the proof of Theorem \ref{thmC}), and so $S^1\times X_2\geq BH_5$. The latter is not possible by the Factorization Lemma \ref{l:KotschickNeofytidisgeneralization}.

\medskip
\noindent{\bf Case II.} Next, suppose that $h(\Gamma_1)=2$. Then $\Gamma_1$ is isomorphic to $\Z^2$ (see the proof of Lemma \ref{l:cd=3nilpotent}). Since $\Gamma_2$ is nilpotent and torsion-free, its center has positive rank. Thus $\mathrm{rank} C(\Gamma_1\times\Gamma_2)\geq3$, which contradicts Lemma \ref{l:extensionscenter} because $\Gamma_1\times\Gamma_2$ fits into the short exact sequence (\ref{eq.BH5}), where $\Gamma_1\cap\Gamma_2$ and $C(H_5)$ are both infinite cyclic.

\medskip
\noindent{\bf Case III.} The only remaining case is when $h(\Gamma_1)=3$. In that case, $h(\Gamma_2)=3$. By Lemma \ref{l:cd=3nilpotent}, each of the groups $\Gamma_i$ is isomorphic to some $G_{n_i}=\langle x,y,z \ \vert \ [x,y]=z^{n_i}, \ xz=zx, \ yz=zy \rangle$, for $n_i \geq 0$. As in Case II, none of the $\Gamma_i$ can be Abelian of rank greater than 1, by Lemma \ref{l:extensionscenter}. We conclude that $n_i \geq 1$. 
According to the proof of Theorem \ref{t:KotschickLoehmain} (cf. Section \ref{s:notation}) and because $BH_5$ is aspherical, there exist two non-trivial homology classes $\alpha_i\in H_{\dim X_i}(B\Gamma_i;\Q)$ such that
\begin{eqnarray*}
H_5(B\varphi)(\alpha_1 \times \alpha_2)=  \deg(f) \cdot [BH_5].
\end{eqnarray*}
Moreover, we know by Case I that one of the $X_i$ must have dimension two and the other dimension three. Without loss of generality, suppose that $\dim X_1=3$ and $\dim X_2=2$. Recall that, for $n \geq 1$, every $G_n$ is realizable by a non-trivial circle bundle over $T^2$. Therefore the cycle $\alpha_1\in H_3(B\Gamma_2;\Q)$ is realized by such a nilpotent $3$-manifold $N$ and the cycle $\alpha_2\in H_2(B\Gamma_2;\Q)$ is realized by $T^2$. This means that there exists a continuous composite map  
\[
T^2\times N\longrightarrow B\Gamma_1\times B\Gamma_2 \longrightarrow BH_5,
\]
which in degree $5$ homology maps the fundamental class $[T^2\times N]$ to a non-trivial multiple of $[BH_5]$. In particular, the product $S^1\times (S^1\times N)$ dominates $BH_5$ which is impossible as we have seen in Case I.

\medskip
After having examined all the possibilities for the groups $\Gamma_1$ and $\Gamma_2$, we conclude that $BH_5$ cannot be dominated by products.
\end{proof}

\bibliographystyle{amsplain}

\begin{thebibliography}{99}

\bibitem{Bieri:book}
R. Bieri, {\sl Homological dimension of discrete groups}, Second edition. Queen Mary College Mathematical Notes, Queen Mary College, Department of Pure Mathematics, London, 1981.

\bibitem{Bowden}
J. Bowden, {\em The topology of symplectic circle bundles}, Trans. Amer. Math. Soc., {\bf 361} (2009), no. 10, 5457--5468.

\bibitem{CWY}
 S. Cappell, S. Weinberger and M. Yan,  {\em Closed aspherical manifolds with center}, J. Topol., {\bf 6} (2013), no. 4, 1009--1018.
 
\bibitem{CarlsonToledo}
J. A. Carlson and D. Toledo, {\em Harmonic mappings of K\"ahler manifolds to locally symmetric spaces}, Inst. Hautes \'Etudes Sci. Publ. Math., {\bf 69} (1989), 173--201.


\bibitem{Eberlein1}
P. Eberlein, {\em Lattices in spaces of nonpositive curvature}, Ann. of Math. (2), {\bf 111} (1980), no. 3, 435--476.

\bibitem{Eberlein2}
P. Eberlein, {\em Isometry groups of simply connected manifolds of nonpositive curvature II}, Acta Math., {\bf 149} (1982), no. 1-2, 41--69.

\bibitem{EilenbergSteenrod}
S. Eilenberg, {\em On the problems of topology}, Ann. of Math. (2), {\bf 50} (1949), 247--260.

\bibitem{Epstein}
D. B. A. Epstein, {\em Factorization of 3-manifolds}, Comment. Math. Helv., {\bf 36} (1961), 91--102.

\bibitem{FarbWeinb}
B.~ Farb and S.~Weinberger, {\em Isometries, rigidity and universal covers}, Ann. of Math., {\bf 186} (2008), no. 3, 915--940.

\bibitem{Filipkiewicz}
R. Filipkiewicz, {\sl Four-dimensional geometries}, PhD thesis, University of Warwick, 1983.

\bibitem{SakamotoFukuhara}
S. Fukuhara and K. Sakamoto, {\em Classification of $T^2$-bundles over $T^2$},  Tokyo J. Math., {\bf 6} (1983), no. 2, 310--327.


\bibitem{GromollWolf}
D. Gromoll and J. A. Wolf, {\em Some relations between the metric structure and the algebraic structure of the fundamental group in manifolds of nonpositive curvature}, Bull. Amer. Math. Soc., 77 (1971), 545--552.

\bibitem{Gromov:bounded}
M. Gromov, {\em Volume and bounded cohomology}, Inst. Hautes \'Etudes Sci. Publ. Math., {\bf 56} (1982), 5--99.

\bibitem{Gromov:metric}
M. Gromov, {\sl Metric Structures for Riemannian and Non-Riemannian Spaces}, With appendices by M. Katz, P. Pansu and S. Semmes, translated from the French by S. M. Bates, Progress in Mathematics 152, Birkh\"auser Verlag, 1999.

\bibitem{Gromov:maps}
M. Gromov, {\em Singularities, expanders and topology of maps. I. Homology versus volume in the spaces of cycles}, Geom. Funct. Anal., {\bf 19} (2009), no. 3, 743--841.

\bibitem{Gruenberg}
K. W. Gruenberg, {\sl Cohomological topics in group theory}, Lecture Notes in Mathematics {\bf 143}, Springer-Verlag, Berlin-New York, 1970.

\bibitem{Hillman}
J. A. Hillman, {\sl Four-manifolds, geometries and knots}, Geometry and Topology Monographs {\bf 5}, Coventry, 2002.
 
\bibitem{KapLeeb}
M. Kapovich and B. Leeb, {\em Actions of discrete groups on nonpositively curved spaces}, Math. Ann., {\bf 306} (1996), no. 2, 341--352.

\bibitem{Kotschick:4-mfds}
D. Kotschick, {\em Remarks on geometric structures on compact complex surfaces}, Topology, {\bf 31} (1992), no. 2, 317--321.

\bibitem{KotschickLoeh1}
D. Kotschick and C. L\"oh, {\em Fundamental classes not representable by products}, J. London Math. Soc., {\bf 79} (2009), 545--561.

\bibitem{KotschickLoeh2}
D. Kotschick and C. L\"oh, {\em Groups not representable by products}, Groups Geom. Dyn., {\bf 7} (2013), 181--204.

\bibitem{KotschickNeofytidis}
D. Kotschick and C. Neofytidis, {\em On three-manifolds dominated by circle bundles}, Math. Z., {\bf 274} (2013), 21--32.

\bibitem{Lueck:book}
W. L\"uck, {\sl $L^2$-invariants: theory and applications to geometry and K-theory}, Springer-Verlag, Berlin, 2002.

\bibitem{Lue}
W. L\"uck, {\em Survey on aspherical manifolds}, European Congress of Mathematics, Eur. Math. Soc., 53--82, 2010.

\bibitem{MilnorThurston}
J. Milnor and W. Thurston {\em Characteristic numbers of 3-manifolds} Enseignement Math. (2), {\bf 23} (1977), no. 3-4, 249--254.

\bibitem{Neofytidis}
C. Neofytidis, {\em Branched coverings of simply connected manifolds}, Topol. Appl., {\bf 178} (2014), 360--371.

\bibitem{Neothesis}
C. Neofytidis, {\sl Non-zero degree maps between manifolds and groups presentable by products}, Munich thesis, 2014, available online at http://edoc.ub.uni-muenchen.de/17204/.

\bibitem{NeoOrder}
C. Neofytidis, {\em Ordering Thurston's geometries by maps of non-zero degree}, J. Topol. Anal. (to appear).

\bibitem{Schur}
I. Schur, {\em \"Uber die Darstellung der endlichen Gruppen durch gebrochene lineare Substitutionen}, Journal f\"ur die reine und angewandte Mathematik, {\bf 127} (1904), 20--50.

\bibitem{Scott:3-mfds}
P. Scott, {\em The geometries of 3-manifolds}, Bull. London Math. Soc., {\bf 15} (1983), 401--487.

\bibitem{Stallings:fiber}
J. Stallings, {\em On fibering certain 3-manifolds}, In Topology of 3-manifolds and related topics (Proc. The Univ. of Georgia Institute, 1961), 95--100, Prentice-Hall, Englewood Cliffs, N.J., 1962.

\bibitem{Sterbel}
R. Strebel, {\em A remark on subgroups of infinite index in Poincar\'e duality groups}, Comment. Math. Helv., {\bf 52} (1977), no. 3, 310--314.

\bibitem{Thom}
R. Thom, {\em Quelques propri\'et\'es globales des vari\'et\'es diff\'erentiables}, Comment. Math. Helv., {\bf 28} (1954), 17--86.

\bibitem{Thurstonbook}
W. P. Thurston, {\sl Three-Dimensional Geometry and Topology}, Princeton University Press, 1997.

\bibitem{Ue1}
M. Ue, {\em Geometric 4-manifolds in the sense of Thurston and Seifert 4-manifolds I}, J. Math. Soc. Japan, {\bf 42} (1990), 511--540.

\bibitem{Wall:geom4-mfds1}
C. T. C. Wall, {\em Geometries and geometric structures in real dimension 4 and complex dimension 2}, In Geometry and Topology (College Park, Md., 1983/84), Lecture Notes in Math., {\bf 1167}, Springer, Berlin, (1985), 268--292.

\bibitem{Wall:geom4-mfds2}
C. T. C. Wall, {\em Geometric structures on compact complex analytic surfaces}, Topology, {\bf 25} (1986), no. 2, 119--153.

\bibitem{Wang:3-mfdsasp}
S. Wang, {\em The existence of maps of non-zero degree between aspherical 3-manifolds}, Math. Z., {\bf 208} (1991), 147--160.

\bibitem{Zhang}
W. Zhang, {\em Geometric structures, Gromov norm and Kodaira dimensions}, Adv. Math. {\bf 308} (2017), 1--35.

\end{thebibliography}

\end{document}